\documentclass[12pt, oneside, reqno]{amsart}
\usepackage[foot]{amsaddr}
\usepackage{amssymb,amsbsy,exscale,amsmath,amsfonts,amssymb,amscd}
\usepackage{geometry}   
\usepackage{comment}
\usepackage{tabularx}
\newcolumntype{C}{>{\centering\arraybackslash}X}
\usepackage{graphicx}
\usepackage{booktabs}
\usepackage{epstopdf}
\usepackage{bm}
\usepackage{multicol}
\usepackage{multirow}
\usepackage{subfigure}
\usepackage{tikz}
\usepackage{enumitem}
\usepackage{array,multirow}
\usepackage{mathabx}
\usepackage{floatrow}            
\newcommand{\vect}[1]{\mathbf{#1}}

\newcommand{\zb}{\vect{z}}

\newcommand{\rb}{\vect{r}}

\newcommand{\mub}{\boldsymbol{\mu}}

\newcommand{\phasespace}{ \mathcal{M} }
\newcommand{\yb}{ \vect{y} }
\newcommand{\gb}{ \vect{g} }
\newcommand{\xb}{ \vect{x} }

\newcommand{\Nhperp}{\vect{N}_h^\perp}

\newcommand{\PK}{\mathcal{P}_k(K)}
\newcommand{\PKb}{\boldsymbol{\mathcal{P}}_k(K)}

\newcommand{\partialt}[1]{\dot{#1}}
\newcommand{\partialtt}[1]{\ddot{#1}}
\newcommand{\divergence}{\mathrm{div}\, }
\newcommand{\grad}{\vect{\mathrm{grad}}\,}

\newcommand{\rot}{\mathrm{rot}}
\newcommand{\curl}{\vect{\mathrm{curl}}}
\newcommand{\Pbrack}[2]{\left\{ #1,\,#2\right\}}
\newcommand{\fder}[2]{\frac{\delta #1}{\delta #2}}

\newcommand{\Thnorm}[1]{\| #1\|_{\mathcal{T}_{h}}}
\newcommand{\pThnorm}[1]{\| #1\|_{\partial\mathcal{T}_{h}}}

\usepackage{stmaryrd}
\newcommand{\jump}[1]{\llbracket #1\rrbracket}
\newcommand{\Th}{\mathcal{T}_{h}}
\newcommand{\Faces}{\mathcal{F}_{h}}
\newcommand{\Facesboundary}{\mathcal{F}_{h}^{\partial}}
\newcommand{\Facesinterior}{\mathcal{F}_{h}^{0}}

\newcommand{\Skeleton}{\partial \mathcal{T}_{h}}

\newcommand{\Kprod}[2]{( #1, #2)_{K}}
\newcommand{\pKprod}[2]{\langle #1, #2\rangle_{\partial K}}

\newcommand{\Thprod}[2]{( #1, #2)_{\mathcal{T}_h}}
\newcommand{\pThprod}[2]{\langle #1, #2\rangle_{\partial \mathcal{T}_{h}}}

\newtheorem{thm}{Theorem}

\newcommand{\velocity}{\vect{u}}
\newcommand{\fb}{\vect{f}}
\newcommand{\displacement}{\vect{w}}
\newcommand{\normal}{\vect{n}}
\newcommand{\Hamiltonian}{\mathcal{H}}
\newcommand{\Jop}{\mathrm{J}}
\newcommand{\Lb}{\vect{L}}
\newcommand{\sigmacheck}{\widecheck{\sigma}}
\newcommand{\phihat}{\widehat{\phi}}

\definecolor{Celeste}{HTML}{0176DE}
\definecolor{Azul}{HTML}{173F8A}
\definecolor{Azuloscuro}{HTML}{03122E}
\definecolor{Amarillo}{HTML}{FEC60D}
\definecolor{Amarillooscuro}{HTML}{E3AE00}
\definecolor{Negro}{HTML}{000000}
\definecolor{GrisA}{HTML}{707070}
\definecolor{GrisB}{HTML}{9C9C9C}
\definecolor{GrisC}{HTML}{C6C6C6}
\definecolor{GrisclaroA}{HTML}{EAEAEA}
\definecolor{GrisclaroB}{HTML}{F0F0F0}
\definecolor{GrisclaroC}{HTML}{F6F6F6}
\definecolor{Blanco}{HTML}{FFFFFF}
\definecolor{Verde}{HTML}{00AA00}
\definecolor{Rojo}{HTML}{F24F4F}

\newcommand{\logLogSlopeTriangleRight}[6]
{
    \pgfplotsextra
    {
        \pgfkeysgetvalue{/pgfplots/xmin}{\xmin}
        \pgfkeysgetvalue{/pgfplots/xmax}{\xmax}
        \pgfkeysgetvalue{/pgfplots/ymin}{\ymin}
        \pgfkeysgetvalue{/pgfplots/ymax}{\ymax}

        \pgfmathsetmacro{\xArel}{#1}
        \pgfmathsetmacro{\yArel}{#3}
        \pgfmathsetmacro{\xBrel}{#1+#2}
        \pgfmathsetmacro{\yBrel}{\yArel}
        \pgfmathsetmacro{\xCrel}{\xArel}

        \pgfmathsetmacro{\lnxB}{\xmin*(1-(#1-#2))+\xmax*(#1-#2)} 
        \pgfmathsetmacro{\lnxA}{\xmin*(1-#1)+\xmax*#1} 
        \pgfmathsetmacro{\lnyA}{\ymin*(1-#3)+\ymax*#3} 
        \pgfmathsetmacro{\lnyC}{\lnyA+#4*(\lnxA-\lnxB)}
        \pgfmathsetmacro{\yCrel}{\lnyC-\ymin)/(\ymax-\ymin)} 

        \coordinate (A) at (rel axis cs:\xArel,\yArel);
        \coordinate (B) at (rel axis cs:\xBrel,\yCrel);
        \coordinate (C) at (rel axis cs:\xCrel,\yCrel);

        \draw[#5]   (A)--
                    (B)-- 
                    (C)-- node[pos=0.5,anchor=east] {#6}
                    cycle;
    }
}

\newtheorem{theorem}{Theorem}[section]
\newtheorem{lemma}{Lemma}[section]

\usepackage{pdfsync}
\usepackage{hyperref}
\usepackage{algpseudocode} 
\usepackage[ruled,vlined]{algorithm2e}
\usepackage{floatrow}
\usepackage{multirow}
\usepackage{wrapfig}
\usepackage{ctable} 
\usepackage{algpseudocode}
\usepackage{flushend}
\usepackage{graphicx}
\usepackage[]{graphicx}
\usepackage{xcolor}
\usepackage{color}
\usepackage{epstopdf}
\usepackage{pifont}
\usepackage{todonotes}
\usepackage{pgfplots}
\usepgfplotslibrary{groupplots}
\usepgfplotslibrary{colorbrewer}
\pgfplotsset{compat=newest}
\DeclareCaptionLabelFormat{AppendixTables}{Table A.#2}
\DeclareCaptionLabelFormat{AppendixTables2}{Table B.#2}
\newsavebox{\measurebox}
\usepackage{lipsum}
\usepackage[skip=0.33\baselineskip]{caption}
\usepackage{adjustbox}
\usepackage{mwe}
\usepackage[most]{tcolorbox}
\usepackage{float}
\title[Symplectic HDG Methods for Linearized shallow water Equations ]{Symplectic Hamiltonian Hybridizable Discontinuous Galerkin Methods for Linearized Shallow Water Equations}
\author[C. N\'unez]{Cristhian  N\'uñez}
\address[1]{Facultad de Matem\'aticas, Pontificia Universidad Cat\'olica de Chile, Chile}
\email{aanunez6@uc.cl}
\author[M. A. S\'anchez]{Manuel A. S\'anchez}
\address[2]{Instituto de Ingenier\'ia Matem\'atica y Computacional, Facultad de Matem\'aticas y Escuela de Ingenier\'ia, Pontificia Universidad Cat\'olica de Chile, Santiago, Chile}
\email{manuel.sanchez@uc.cl}
\thanks{C. Nuñez was supported by Doctorado Nacional Scholarship 2021, ANID Chile.}
\thanks{M.A. S\'anchez was partially supported by FONDECYT Regular grant N. 1221189 and by Centro Nacional de Inteligencia Artificial CENIA, FB210017, Basal ANID Chile.}
\date{} 
\begin{document}
\maketitle
\begin{abstract}
This paper focuses on the numerical approximation of the linearized shallow water equations using hybridizable discontinuous Galerkin (HDG) methods,  leveraging the Hamiltonian structure of the evolution system. First, we propose an equivalent formulation of the equations by introducing an auxiliary variable. Then, we discretize the space variables using HDG methods, resulting in a semi-discrete scheme that preserves a discrete version of the Hamiltonian structure. The use of an alternative formulation with the auxiliary variable is crucial for developing the HDG scheme that preserves this Hamiltonian structure. The resulting system is subsequently discretized in time using symplectic integrators, ensuring the energy conservation of the fully discrete scheme. We present numerical experiments that demonstrate optimal convergence rates for all variables and showcase the conservation of total energy, as well as the evolution of other physical quantities.
\end{abstract}
\section{Introduction}\label{s:Introduction} 
In this paper, we explore numerical methods that combine hybridizable discontinuous Galerkin (HDG) methods with symplectic time integrators to approximate solutions of the linearized shallow water equations. The shallow water equations play an important role in the modeling of geophysical fluid flows \cite{20_MR1718369}, including tsunami simulations \cite{MR2867413}, hurricane-induced storm surge prediction \cite{Akbar2013}, tidal stream simulation \cite{Dyke1987}, and large-scale atmospheric dynamics \cite{Spassova2011}. Several approaches have been proposed to solve this problem, including finite difference \cite{MR1033901,Arakawa1997},  finite volume \cite{MR1784946,topography22}, finite element \cite{MR3292650,11_ThanBui} methods, and more recently HDG methods \cite{Bui-Thanh2016,10_Dawson}.

HDG methods, originally developed in \cite{eliptichdg} for diffusion problems, offer high-order accuracy, geometric flexibility for unstructured meshes, and local conservation properties of discontinuous Galerkin schemes,  while leading to a significant reduction in globally coupled unknowns through hybridization \cite{eliptichdg4}. These methods have been applied to a variety of partial differential equations, including Stokes equations \cite{article0,article1,article3}, wave equations \cite{MR3643828,MR3452794}, and Maxwell's equations \cite{MR2822937,MR4092329}, to name a few. In the case of the shallow water equations, \cite{Bui-Thanh2016} introduces four types of HDG methods for the linearized model, differing in the definition of their numerical fluxes. Two of these fluxes are derived from the Rankine–Hugoniot condition and a simplified variant, while the third and fourth are based on Lax–Friedrichs fluxes. These four methods are dissipative, which is a drawback in applications where energy conservation is important \cite{3_Kieran}.  On the other hand, \cite{10_Dawson} compares two HDG methods for the nonlinear model: an explicit scheme that solves equations locally on each edge without global coupling, and an implicit scheme that yields a globally coupled nonlinear system. Both achieve optimal convergence in numerical experiments, each offering distinct advantages in terms of computational efficiency and stability. However, this comparison does not address energy conservation.

Hamiltonian systems are a class of dynamical systems whose evolution is governed by a Hamiltonian functional and a Poisson bracket. These systems preserve the energy and the symplectic structure \cite{Olver93book}. In this work, we focus on a particular instance of Hamiltonian systems: the shallow water equations, in which energy conservation plays a crucial role in various applications. However, standard numerical methods usually may suffer energy drift in long-time simulations (see, e.g., \cite{XuvanderVegtBokhove08}). This motivated the development of structure-preserving discretization methods, including finite difference schemes that conserve energy and enstrophy \cite{Arakawa1981, Arakawa1997}, as well as bracket-based discretizations using Poisson and Nambu brackets \cite{Salmon, Salmon2, Salmon2009}.  In the context of finite elements, \cite{XuvanderVegtBokhove08} introduces discontinuous Hamiltonian finite element methods based on a Poisson bracket formulation for linear hyperbolic systems, including shallow-water equations. These methods exhibit improved long-term stability when combined with symplectic time integrators. Moreover, \cite{1_Kieran} develops an entropy-stable discontinuous Galerkin method for the vector invariant shallow water equations on the sphere, ensuring the conservation of mass, vorticity, and geostrophic balance. In a related work, \cite{3_Kieran} proposes a mixed compatible finite element framework for thermal atmospheric dynamics that conserves both energy and entropy without entropy damping.  Further developments by  \cite{12_Cotter,11_Cotter_MR3874575,MR4038160}  studies compatible finite element methods by a conserving Hamiltonian formulation in both planar and spherical geometries. Meanwhile,  \cite{5_GASSNER2016291} and \cite{6_2_WINTERMEYER2018447} propose entropy-stable discontinuous Galerkin methods that also satisfy well-balanced and positivity-preserving properties.

In this paper, we propose numerical methods based on a reformulation of the governing equations, resulting in an equivalent system that admits a discrete Hamiltonian structure under HDG discretization. To ensure the conservation of total energy, we discretize time using symplectic timeintegrators. This approach builds upon previous work developed for linear wave propagation \cite{SanchezCiucaNguyenPeraireCockburn2017}, acoustic waves \cite{CockburnFuHungriaSanchezSayas2018}, linear elastodynamics \cite{symplectic1}, electromagnetics \cite{Symplectic_Hamiltonian2,MR4609879}, semilinear wave models \cite{joaquin} and the general framework for linear Hamiltonian systems \cite{CockburnDuSanchez2023}. Specifically, we focus on the linearized shallow water equations, originally introduced in \cite{Bui-Thanh2016,GiraldoWarburton2008},  on a two-dimensional bounded polyhedral domain $\Omega$
\begin{subequations}\label{eq:LSWEs}
\begin{alignat}{4}
\label{eq:LSWEs:a}
\partialt{\phi}&=- \divergence (\Phi \velocity), & &\quad \text{in $\Omega\times(0,T],$}  \\ 
\label{eq:LSWEs:b}
\Phi \partialt{\velocity} & =-\Phi \grad \phi + f (\Phi \velocity)^{\perp} - \gamma (\Phi \velocity)-\vect{s},& & \quad\text{in $\Omega\times(0,T],$}  
\end{alignat}
where $\phi = \phi(\xb, t)$ denotes the geopotential height, defined by $\phi = g h$, with $g$ the gravitational acceleration and $h$ the perturbation of the free surface elevation. The vector field $\velocity = \velocity(\xb, t) = [u_1(\xb, t), u_2(\xb, t)]^\top$ represents the horizontal velocity perturbation, and the operator $(\cdot)^{\perp}$ is defined by $\velocity^\perp := [u_2, -u_1]^\top$. The parameter $\Phi > 0$ denotes the constant mean geopotential height of the background flow, $\gamma \geq 0$ is a bottom friction coefficient, $\vect{s}$ is a external force, that can represents the wind effect or bathymetry influence, $f = f_0 + \beta(y - y_m)$ is the Coriolis parameter, with $f_0$, $\beta$, and $y_m$ being prescribed constants.

The initial conditions for the geopotential height and the horizontal velocity are given by
\begin{alignat}{4}\label{eq:LSWEs:c}
\phi(0) = \phi_0, \quad \velocity(0) = \velocity_0, \quad \text{in } \Omega \times {\{t = 0\}}.
\end{alignat}
The initial-boundary value problem is completed with suitable boundary conditions. Here, we consider wall-type boundary condition on the domain boundary $\Gamma := \partial \Omega$
\begin{alignat}{4}\label{eq:LSWEs:d}
\velocity \cdot \normal = 0, \quad \text{on } \Gamma \times (0, T].
\end{alignat}
\end{subequations}
\begin{table}
\caption{Glossary of physical quantities for the shallow water equations}
\centering
\begin{tabular}{@{\hskip .1in}l@{\hskip .3in}c@{\hskip .5in}l@{\hskip .1in}}
\toprule
Name & Symbol & Definition  \\
\midrule
Mass  & $\textit{m}$ & $\phi$ \\[0.25em]
Momentum  &  $\boldsymbol{p}$ & $\Phi \velocity$ \\[0.25em]
Potential energy  & $ P $ & $ \frac{1}{2}\phi^2 $ \\[0.25em]
Kinetic energy  & $ K $ & $\frac{1}{2}\Phi \velocity^2$  \\[0.25em]
Total energy & $\mathcal{E}$ & $ P + K $ \\[0.25em]
Vorticity & $\zeta$ & $ \displaystyle \rot \, \velocity $ \\[0.25em]
Angular momentum & $L$ & $ \displaystyle  \boldsymbol{x}^\perp \cdot (\Phi \velocity) $ \\[0.25em]
Potential vorticity & $\theta$ & $ \displaystyle \Phi (\rot \, \velocity)-f\Phi^{-1} \phi $ \\[0.25em]
Potential enstrophy & $E_n$ & $ \displaystyle  \Phi (\rot \, \velocity)^2 $ \\
\bottomrule 
\end{tabular}
\label{tab:GlossarySW}
\end{table}
Table~\ref{tab:GlossarySW} provides a summary of some important physical quantities associated with the linearized shallow water equations. These include mass, linear momentum, potential and kinetic energy, as well as total energy. Rotation quantities such as vorticity, angular momentum, potential vorticity, and potential enstrophy are also listed. These quantities are fundamental in the analysis of conservation properties; see \cite{20_MR1718369, Essentials} for additional details.

We refer to the system \eqref{eq:LSWEs} as the $(\phi-\velocity)$ formulation. Following the methodology studied in \cite{SanchezCiucaNguyenPeraireCockburn2017,  symplectic1, Symplectic_Hamiltonian2,joaquin}, we propose an alternative but equivalent formulation by introducing an auxiliary variable $\displacement$. Specifically, we consider the following reformulation
\begin{subequations}\label{eq:LSWEs2}
\begin{alignat}{4}
\label{eq:LSWE2:a}
\partialt{\displacement} & =  \Phi \velocity, & & \quad \text{in }\Omega\times(0,T], \\ 
\label{eq:LSWE2:b}
\Phi\partialt{\velocity} & =-\Phi\grad \phi +f\Phi\velocity^{\perp} - \gamma \Phi\velocity-\vect{s} ,& &  \quad \text{in }\Omega\times(0,T],\\ 
\label{eq:LSWE2:c}
\phi& =-   \divergence  \displacement,& &\quad   \text{in }\Omega\times(0,T].
\end{alignat}
We observe that in \eqref{eq:LSWE2:c}, the auxiliary variable $\displacement$ satisfies a steady-state relation with the geopotential height $\phi$. The reformulated problem is completed with initial conditions
\begin{alignat}{4}
\label{eq:LSWE2:d} 
\displacement(0) = \displacement_0, \quad \velocity(0) = \velocity_0, \quad \text{in } \Omega \times \{t = 0\},
\end{alignat}
and the wall-type boundary condition
\begin{alignat}{4}
\label{eq:LSWEe}
\displacement \cdot \normal = 0, \quad \text{on } \Gamma \times (0, T].
\end{alignat}
\end{subequations}
We refer to the system \eqref{eq:LSWEs2} as the ($\velocity-\displacement$) formulation. Observe that \eqref{eq:LSWEs2} can be written as a second-order equation for $\displacement$ (a vector wave type equation) as
\begin{alignat*}{4}
\partialtt{\displacement} - \Phi \grad(\divergence \displacement)
- f \, \partialt{\displacement}^{\perp} + \gamma \, \partialt{\displacement} +\vect{s} = 0,
\end{alignat*}
subject to the initial conditions $\displacement(0) = \displacement_0$ and $\dot{\displacement}(0) = \Phi \velocity_0$.  

The paper is organized as follows. In Section~\ref{sec:section2}, we present the two Hamiltonian formulations associated with linearized shallow water equations. Section~\ref{sec:section3} introduces the spatial discretization using HDG methods. In Section~\ref{sec:section4}, we show that the resulting semi-discrete system defines a Hamiltonian system of ordinary differential equations and establish the corresponding conservation laws. Section~\ref{sec:section5} details the construction of symplectic Hamiltonian HDG methods based on the $(\velocity-\displacement)$ formulation. In Section~\ref{sec:section6}, we present numerical experiments that validate the convergence rates and conservation properties of the proposed approach. Finally, Section~\ref{sec:section7} discusses the future work and conclusions.
\section{The Hamiltonian structure of linear shallow water equations}\label{sec:section2}
In this section, we present the Hamiltonian structure of the standard $(\phi-\velocity)$ formulation \eqref{eq:LSWEs}, and the alternative $(\velocity-\displacement)$ formulation \eqref{eq:LSWEs2}.  In what follows, unless otherwise stated, we assume that the parameters $\gamma=0$, $\vect{s}=0$ and $\Phi$, $g,$ $f$ are independent of time.

\subsection{Preliminaries}
We now recall the definition of a Hamiltonian system of partial differential equations, following \cite{Olver93book,McLachlan94}. Given a triplet $(\mathcal{M}, \Hamiltonian, \Pbrack{\cdot}{\cdot})$, where $\phasespace$ is a phase-space of smooth functions with suitable conditions, $\Hamiltonian$ is a Hamiltonian functional, and a Poisson bracket $\Pbrack{\cdot}{\cdot}$, a Hamiltonian partial differential equation (PDE) system seeks a function $u\in \phasespace$ satisfying
\begin{equation}\label{eq:HPDE}
\partialt{u} = \Pbrack{u}{\Hamiltonian}(u) = \mathrm{J}\fder{\Hamiltonian}{u}.
\end{equation}
The Poisson bracket $\Pbrack{\cdot}{\cdot}$ is a bilinear and anti-symmetric form defined for functionals $F=F[u]$ and $G=G[u]$ on the phase-space $\mathcal{M}$, induced by an anti-symmetric structure operator $\Jop$ and defined by
\begin{equation}\label{eq:genPbracket}
    \Pbrack{F}{G}(u ) : = \int_{\Omega} \fder{F}{u} \Jop \fder{G}{u}.
\end{equation}
In addition, the definition of the Poisson brackets requires that it satisfies the Jacobi identity, $\{\{F, G\}, H\} + \{ \{H, F\},  G\} + \{ \{  G,  H\},F\} = 0$. For a constant (variable independent) structure operator $\mathrm{J}$, this property is immediately satisfied, see \cite{Olver93book} for more details.  In the following section, we will show that the formulations described by \eqref{eq:LSWEs} and \eqref{eq:LSWEs2} can be written in the structure \eqref{eq:HPDE}. We will present the triplet associated to each formulation, specifying the phase-space, the Poisson bracket, and the Hamiltonian functional.

Standard notation for the functional spaces and differential operators is introduced next; see, for instance, \cite{girauld}. We denote the vector-valued space $\boldsymbol{L}^2(\Omega) := [L^2(\Omega)]^2$ and introduce the following functional spaces
\begin{alignat*}{4}
H(\operatorname{rot} )&:=\{\zb\in \Lb^2(\Omega): \rot \, \zb \in L^2(\Omega) \},& \quad H(\divergence)&:=\{\zb\in \Lb^2(\Omega): \divergence \zb \in L^2(\Omega) \},\\
\mathring{H}(\operatorname{rot})&:=\{\zb\in H(\operatorname{rot}): \zb\cdot \normal^\perp=0\},
& \quad \mathring{H}(\divergence)&:=\{\zb\in H(\divergence): \zb \cdot \normal=0 \},
\end{alignat*}
and the differential operators
$$
\begin{aligned}
\grad \varphi:=(\frac{\partial \varphi}{\partial x}, \frac{\partial \varphi}{\partial y}), \quad \curl \,\varphi & :=(\frac{\partial \varphi}{\partial y},-\frac{\partial \varphi}{\partial x}), \\ \divergence \zb:=\frac{\partial z_{1}}{\partial x}+\frac{\partial z_{2}}{\partial y}, \quad
\operatorname{rot} \zb & :=\frac{\partial z_{2}}{\partial x}-\frac{\partial z_{1}}{\partial y},
\end{aligned}
$$
for a smooth scalar and vector functions $\varphi$  and $\zb=[z_1,z_2]^\top$.
\subsection{Hamiltonian form of the geopotential height - velocity ($\phi-\velocity$)  formulation} \label{ss:phi-uform}
We consider the system \eqref{eq:LSWEs}, which is equivalent to a Hamiltonian partial differential equation characterized by the triplet $(\phasespace_1, \Hamiltonian_1, \Pbrack{\cdot}{\cdot}_{1})$, where the phase space for the variables $(\phi,\velocity)$ is defined as $\phasespace_1 := L^2(\Omega)\times \mathring{H}(\divergence)$. The corresponding Hamiltonian functional $\Hamiltonian_1$ and Poisson bracket $\Pbrack{\cdot}{\cdot}_1$ are given by
\begin{alignat*}{4}
     \Hamiltonian_{1}[\phi,\velocity]=\frac{1}{2}\int_\Omega \left(\phi^{2} +  \Phi |\velocity|^{2}\right),
\end{alignat*} 
and 
\begin{alignat*}{4}
    \Pbrack{F}{G}_{1} &:=&-\int_\Omega \left(\fder{F}{\phi} \divergence \fder{G}{\velocity} + \fder{F}{\velocity}\cdot \grad \fder{G}{\phi} - \frac{f}{\Phi} \left(\fder{G}{\velocity}\right)^{\perp}\cdot\fder{F}{\velocity}\right),
\end{alignat*}
defined for functionals $F=F[\phi, \velocity]$ and $G=G[\phi, \velocity]$. Note that $\{\cdot,\cdot\}_1$ is indeed a Poisson bracket since it arises from \eqref{eq:genPbracket} with anti-symmetric structure operator 
\begin{alignat*}{4}
\mathrm{J}_1 = -
\begin{bmatrix}
    0 &  \divergence \\
   \grad & -f \Phi^{-1}(\cdot)^{\perp}
\end{bmatrix}.
\end{alignat*}
The proof of the statement is established by demonstrating the equivalence between formulation \eqref{eq:LSWEs} and $\partialt{F}_1 = \{F_1, \Hamiltonian_1\}_1$. This holds for a suitable \emph{coordinate functional} $F_1$, which is defined by
\begin{alignat*}{4}
F_{1}[\phi, \velocity] = \int_{\Omega} (\phi \varphi + \Phi \velocity\cdot\zb),\quad \text{for } (\varphi, \zb) \in L^2(\Omega)\times \mathring{H}(\divergence).
\end{alignat*}
A simple computation shows that the variational derivatives of the coordinate and Hamiltonian functionals are
\begin{alignat*}{4}
\frac{\delta F_1}{\delta \phi} = \varphi,\quad \frac{\delta F_1}{\delta \velocity} = \Phi \zb,\quad 
\frac{\delta \Hamiltonian_1}{\delta \phi} = \phi,\quad \frac{\delta \Hamiltonian_1}{\delta \velocity} = \Phi \velocity, 
\end{alignat*}
and then, by replacing these in the bracket's definition, we obtain that
\begin{align*}
\int_{\Omega}(\partialt{\phi}\varphi + \Phi \partialt{\velocity}\cdot \zb) &= \partialt{F}_{1}[\phi, \velocity] = \{F_1, \Hamiltonian_1\}_1 \\ & = 
-\int_\Omega \left(\fder{F_1}{\phi} \divergence \fder{\Hamiltonian_1}{\velocity} + \fder{F_1}{\velocity}\cdot \grad \fder{\Hamiltonian_1}{\phi} - \frac{f}{\Phi} \left(\fder{\Hamiltonian_1}{\velocity}\right)^{\perp}\cdot\fder{F_1}{\velocity}\right) \\
& = -\int_\Omega \left(\varphi\divergence \Phi\velocity+ \zb\cdot\Phi \grad \phi- f \left(\Phi\velocity\right)^{\perp}\cdot \zb\right)
\end{align*}
This proves the Hamiltonian structure of the formulation \eqref{eq:LSWEs} for the $(\phi,\velocity)$ variables.
\subsection{Hamiltonian form of the velocity - flux vector field  $(\velocity-\displacement)$ formulation} \label{ss:u-wform}
The system \eqref{eq:LSWEs2} evolves in time the velocity variable $\velocity$ and the auxiliary variable  $\displacement$, satisfying $\phi = -\divergence\displacement$.  Thus, this system  is equivalent to a Hamiltonian PDE with triplet $(\phasespace_2, \Hamiltonian_2, \Pbrack{\cdot}{\cdot}_2)$, where the phase-space of the variables $(\velocity,\displacement)$ is taken as $\phasespace_2 = \Lb^2(\Omega)\times \ring{H}(\divergence)$, the Hamiltonian functional $\Hamiltonian_2$, and the Poisson bracket $\Pbrack{\cdot}{\cdot}_2$ are given by
\begin{alignat}{4}  \label{eq:Hamiltonian}
\Hamiltonian_{2}[\velocity,\displacement]=\frac{1}{2}\int_\Omega \left( |\divergence \displacement|^2+  \Phi |\velocity|^2\right)
\end{alignat}
and
\begin{alignat*}{4}
\Pbrack{F}{G}_{2} =\int_\Omega \left(\fder{F}{\displacement} \cdot \fder{G}{\velocity} - \fder{F}{\velocity}\cdot \fder{G}{\displacement}   +\frac{f}{\Phi}\left(\fder{G}{\velocity}\right)^{\perp}\cdot \fder{F}{\velocity}\right).
\end{alignat*}
Note that in this case, the bracket is defined for functionals $F = F[\velocity,\displacement]$ and $G=G[\velocity,\displacement]$ and it is also clearly a Poisson bracket since the following anti-symmetric structure operator induces it
\begin{alignat*}{4}
\mathrm{J}_2 = 
\begin{bmatrix}
    f \Phi^{-1}(\cdot)^{\perp} & -\mathrm{Id} \\
    \mathrm{Id} & 0
\end{bmatrix}.
\end{alignat*}
We establish now that the formulation \eqref{eq:LSWEs2} is equivalent to $\partialt{F}_2 = \{F_2,\Hamiltonian_2\}_2$ for the coordinate functional
\begin{alignat*}{4}
F_2[\velocity,\displacement]:= \int_{\Omega} \left(  \Phi\velocity\cdot \zb +\displacement\cdot\rb \right), \quad \text{for } (\zb,\rb)\in \Lb^2(\Omega)\times \ring{H}(\divergence).
\end{alignat*}
Observe that the variational derivative of the coordinate and Hamiltonian functional $\Hamiltonian_2$ are
\begin{alignat*}{4}
\fder{F_2}{\velocity} = \Phi\zb,\quad \fder{F_2}{\displacement} = \rb,\quad 
\fder{\Hamiltonian_2}{\velocity} = \Phi \velocity, \quad \fder{\Hamiltonian_2}{\displacement} = -\grad \divergence \displacement,
\end{alignat*}
and thus
\begin{align*}
 \int_{\Omega} \left(\partialt{\displacement}\cdot\rb + \Phi\partialt{\velocity}\cdot \zb\right)
 & = \partialt{F}_2[\velocity,\displacement] = \{F_2,\Hamiltonian_2\}_2 \\
 & = \int_\Omega \left(\fder{F_2}{\displacement} \cdot \fder{\Hamiltonian_2}{\velocity} - \fder{F_2}{\velocity}\cdot \fder{\Hamiltonian_2}{\displacement}   +\frac{f}{\Phi}\left(\fder{\Hamiltonian_2}{\velocity}\right)^{\perp}\cdot \fder{F_2}{\velocity}\right) \\
 & = \int_\Omega \left(\rb\cdot \Phi\velocity + \Phi\zb\cdot (\grad \divergence \displacement)   +\frac{f}{\Phi}\left(\Phi \velocity \right)^{\perp}\cdot \Phi\zb \right) \\
 & = \int_\Omega \left(\rb\cdot \Phi\velocity - \zb\cdot (\Phi \grad \phi)+f\left(\Phi \velocity \right)^{\perp}\cdot\zb \right),
\end{align*}
for $\phi = -\divergence \displacement$. This proves the Hamiltonian structure of the formulation \eqref{eq:LSWEs2}.
It is crucial to highlight some key differences between the two Hamiltonian formulations. Despite the equivalence of the Hamiltonian functionals $\Hamiltonian_1$ and $\Hamiltonian_2$, the formulations differ in their phase-spaces and Poisson brackets.
\subsection{Conservation laws}
One of the objectives of this study is to investigate the conservation of physical quantities relevant to the linearized shallow water equations. This section outlines the conservation laws derived from the equations, as summarized in Table~\ref{tab:ConservationLaws}, and discusses their connection to the Hamiltonian formulation. To illustrate the results shown in the table, we focus on the conservation law of vorticity, defined as $\zeta = \rot\, \velocity$, which has a flux given by $f\velocity$ and a zero source term. This leads to the conservation law
\begin{alignat*}{4}
\partialt{\zeta} + \divergence (f \velocity) = 0,
\end{alignat*}
which can be obtained directly from the original system by applying the $\rot(\cdot)$ operator to the momentum equation \eqref{eq:LSWEs:b}. Equivalently, within the Hamiltonian framework, we can \emph{re-discover} this conservation law. Consider the Hamiltonian system $(\phasespace_1, \Hamiltonian_1, \Pbrack{\cdot}{\cdot}_1)$. By selecting the functional
\[
F_{\zeta}[\phi, \velocity] := \int_\Omega \rot(\velocity) \, \varphi = \int_\Omega \zeta \, \varphi,
\]
for a test function $\varphi \in \mathcal{C}^{\infty}_0(\Omega)$, and evaluating in the bracket, we obtain a weak form of the vorticity conservation law
\begin{alignat*}{4}
    \int_{\Omega}\partialt \zeta \varphi=\partialt F_{\zeta}[\phi, \velocity]=\{F_{\zeta},\Hamiltonian_1\}_{1}=-\int_{\Omega}\left(\curl(\varphi)\grad \phi-\frac{f}{\Phi}(\Phi \velocity)^\perp \curl(\varphi) \right)=-\int_\Omega \divergence (f \velocity )\varphi.
\end{alignat*}
Similar proofs for the conservation laws listed in Table \ref{tab:ConservationLaws} can be obtained by selecting the appropriate functionals associated with the physical quantities. In Table \ref{tab:ConservationLawsHamiltonian}, we present the functionals and bracket calculations between the corresponding functional and the Hamiltonian $\Hamiltonian_1$.
Alternatively, we can also obtain the conservation laws under the Hamiltonian system $(\phasespace_2, \Hamiltonian_2, \Pbrack{\cdot}{\cdot}_2)$. For example, consider the functional
\[
G_{\zeta}[\velocity, \displacement] := \int_\Omega \rot(\velocity) \, \varphi = \int_\Omega \zeta \, \varphi,
\]
and thus
\begin{alignat*}{4}
    \int_{\Omega}\partialt \zeta \varphi=\partialt G_{\zeta}[\velocity,\displacement]=\{G_{\zeta},\Hamiltonian_2\}_{2}&=\int_{\Omega}\left( \curl (\varphi)  \grad(\divergence \displacement)+\frac{f}{\Phi}(\Phi \velocity)^\perp \curl(\varphi) \right)\\
    &=-\int_\Omega \divergence (f \velocity )\varphi.
\end{alignat*}
\begin{table}
\centering
\begin{tabular}{@{\hskip .1in}l@{\hskip .4in}l@{\hskip .2in}l@{\hskip .2in}l@{\hskip .1in}}
\toprule
Conservation of & $\eta$ & $f_\eta$ & $S_{\eta}$  
\\[0.25em]
\midrule
Mass & $\phi$ &  $\Phi \velocity$ & 0
\\[0.25em]
Linear momentum& $\Phi \velocity$ &  $(\Phi \phi) \mathrm{Id}$ & $ f(\Phi\velocity)^{\perp}$
\\[0.25em]
Vorticity & $\zeta$ &  $f \velocity$ & 0
\\[0.25em]
Angular momentum &  $\vect{x}^\perp\cdot (\Phi\velocity)$ &  $\vect{x}^{\perp} \cdot (\Phi\phi) \mathrm{Id} $ & $f\Phi \vect{x}\cdot\velocity$
\\[0.25em]
Potential vorticity& $\Phi {\zeta}-f\phi$ & $0$ & $0$
\\[0.35em]
\midrule
\multicolumn{2}{c} \textbf{(Case $f=0$)}
\\[0.25em]
\midrule
Potential enstrophy& $\Phi \zeta^2$ &  $0$ & $0$ 
\\\bottomrule 
\end{tabular}
\caption{Conservation laws for the physical quantity $\eta$ in system~\eqref{eq:LSWEs} under the assumptions $\gamma = 0$ and $\vect{s} = 0$, where the evolution is governed by $\partialt \eta + \divergence f_\eta = S_\eta$. Here, $f_\eta$ denotes the flux associated with $\eta$, and $S_\eta$ represents its source term.}
\label{tab:ConservationLaws}
\end{table}

\begin{table}
\centering
\begin{tabular}{@{\hskip .1in}l@{\hskip .4in}l@{\hskip .4in}l@{\hskip .1in}}
\toprule
Conservation of & Functional $F$ & Poisson bracket $\Pbrack{F}{\Hamiltonian_1}_1$  \\[0.25em]
\midrule
Mass & $\int_\Omega \phi \varphi$ &  $-\int_\Omega \divergence (\Phi \velocity) \varphi$
\\[0.25em]
Linear momentum& $\int_\Omega \Phi \velocity \cdot \zb$ &  $-\int_\Omega (\Phi \grad  \phi \cdot\zb  - f (\Phi\velocity)^\perp\cdot \zb )$ 
\\[0.25em]
Vorticity & $\int_\Omega \zeta \varphi$ &  $\int_\Omega  -\divergence(f \velocity) \varphi$ 
\\[0.25em]
Angular momentum&  $\int_\Omega \vect{x}^\perp \cdot(\Phi\velocity) \varphi$ &  $-\int_\Omega ( (\vect{x}^\perp\cdot \Phi\grad \phi)\varphi - f \vect{x} \cdot (\Phi \velocity) \varphi )$ 
\\[0.25em]
Potential vorticity& $\int_\Omega( \Phi {\zeta}-f\phi)\varphi $ & $\phantom{-}0$ 
\\[0.25em]
\midrule
\multicolumn{2}{c} \textbf{(Case $f=0$)}
\\[0.25em]
\midrule
Potential enstrophy & $\int_\Omega \Phi \zeta^2 \varphi$ &  $\phantom{-}0$ 
\\\bottomrule
\end{tabular}
\caption{Conservation laws, associated functionals, and the corresponding Poisson brackets with the Hamiltonian $\Hamiltonian_1$, computed under the assumptions $\gamma = 0$ and $\vect{s} = 0$. The test functions are $\varphi \in \mathcal{C}_0^\infty(\Omega)$ and $\zb \in \boldsymbol{\mathcal{C}}_0^\infty(\Omega)$.}
\label{tab:ConservationLawsHamiltonian}
\end{table}
\section{Semi-discrete hybridizable discontinuous Galerkin methods}\label{sec:section3}
In this section, we introduce the HDG method for the $(\velocity-\displacement)$ formulation. Here, both variables evolve in time, while the geopotential height $\phi$ is implicitly defined through  $\phi=-\divergence \displacement$.  We also detail the initialization procedure, which now incorporates the variable $\displacement$. Finally, we compare this approach with the HDG method introduced in~\cite{Bui-Thanh2016}, which is based on the $(\phi-\velocity)$ formulation \eqref{eq:LSWEs}, highlighting the key differences between both methods.
\subsection{HDG Notation}\label{ss:Notation_hdg} 
Let $\Th$ be a quasi-uniform mesh \cite{MR1930132} consisting of disjoint elements that partition $\overline{\Omega}$, with mesh size parameter $h$, the maximum inner diameter among all elements in $\Th$. Let $\Faces$ denote the set of all edges, with $\Facesinterior$ and $\Facesboundary$ representing the subsets of the interior and boundary edges, respectively. For an element $K \in \Th$, we denote by $\partial K$ the set of edges that form its boundary, and then $\Skeleton$ denotes the union of all the boundaries of the elements in the mesh $\Th$, and is commonly called the mesh skeleton.

Let $\Kprod{\cdot}{\cdot}$ and $\pKprod{\cdot}{\cdot}$ be the $L^2$--inner products over $K$ and $\partial K$, respectively.    
In addition, the piecewise $L^2$--inner products are defined as
\[
\Thprod{\cdot}{\cdot}:=\sum_{K\in \Th} \Kprod{\cdot}{\cdot}, \quad
\pThprod{\cdot}{\cdot} :=\sum_{K\in \Th}\pKprod{\cdot}{\cdot}.
\]
Analogous definitions apply to vector-valued functions. The norms induced by the inner products above are denoted by $\Thnorm{\cdot}$ and $\pThnorm{\cdot}$. We also introduced a notation for the weighted norms induced by the inner products above
\[
\|\cdot\|_{\Th,c}^{2} = \Thprod{c \cdot}{\cdot}, \quad \|\cdot\|_{\partial \Th, a}^2 =\pThprod{a \cdot}{\cdot}, 
\]
for $c,a\geq 0.$
Finally, we introduce the following discontinuous finite-dimensional approximation spaces.  First, we introduce $\mathcal{P}_{k}(A)$  and $\boldsymbol{\mathcal{P}}_{k}(A)$  as the scalar and vector polynomial spaces of degree less than or equal to $k\geq 0$ respectively, defined over the domain $A$. Then, we define the following scalar and vector volumetric approximation function spaces
\begin{alignat*}{4}
{W}_{h}&:=&\left\{ \varphi \in L^{2}(\Omega): \;\varphi |_{K} \in \PK, \;\forall K \in \mathcal{T}_{h}\right\},\\
\vect{V}_{h}&:=&\left\{ \zb \in \Lb^{2}(\Omega):\;\zb|_{K} \in  \PKb, \;\forall K \in \mathcal{T}_{h}\right\},
\end{alignat*}
and the approximation spaces for the numerical traces 
\begin{alignat*}{4}
M_{h}&:=\left\{ \mu \in L^{2}(\Faces) :\mu|_{F} \in \mathcal{P}_k(F)\; \text{for }F\in \Faces\right\},\\
\vect{M}_h&:= \{ \mub\in \Lb^2(\Faces)\, : \, \mub|_{\partial K}\in \boldsymbol{\mathcal{P}}_k(\partial K)\}, \\
\Nhperp&:= \{ \mub\in \vect{M}_h:   \mub \cdot\normal=0  \text{ on $\Faces$} \text{ and } \jump{\mub \cdot \normal^\perp}=0 \text{ on $\Facesinterior$} \},
\end{alignat*}
where the jumps of  vector function $\mub$, in the tangential directions across the interior edge $\partial K^{-}\cap \partial K^{+}=:F \in \Facesinterior$, for $K^{-}, K^{+}\in \Th$ is given by
\begin{alignat*}{4}   
    \jump{\mub \cdot \normal^\perp}&:=\mub^{+}\cdot (\normal^\perp )^+ +\mub^{-}\cdot (\normal^\perp)^{-},
\end{alignat*}
where $\normal^{\pm}$ are the outward unit normals to $K^{\pm}$, and  $\varphi^{\pm}=\varphi|_{K^{\pm}}$ and $\mub^{\pm}=\mub|_{K^{\pm}}$. 
\subsection{HDG scheme for $(\velocity-\displacement)$ formulation} \label{ss:HDGvw}
In this section, we consider the system \eqref{eq:LSWEs2} and approximate the velocity and the auxiliary variable by $\velocity_{h}\in \vect{V}_h$ and $\displacement_h\in \vect{V}_h$, respectively, and the geopotential height and its trace on the skeleton by $\phi_h\in W_h$ and $\phihat_h\in M_h$, respectively. The approximations satisfy the following evolution system
\begin{subequations}\label{eq:HDG-velocity-w}
\begin{alignat}{4}
\label{eq:HDG-velocity-w:a}
\Thprod{\Phi\partialt{\velocity}_h}{\zb}&= \Thprod{ \phi_h}{\divergence (\Phi\zb)}-\pThprod{\phihat_h}{\Phi\zb\cdot\normal} + \Thprod{f \Phi \velocity_{h}^\perp }{\zb}-\Thprod{\vect{s} }{\zb}, &   \quad \forall \zb \in \vect{V}_h, \\ 
\label{eq:HDG-velocity-w:b}
\Thprod{\partialt{\displacement}_h}{\rb} & = \Thprod{\Phi\,\velocity_h}{\rb},&  \quad \forall\rb \in \vect{V}_h,  \\ 
\label{eq:HDG-velocity-w:c}
\Thprod{\phi_h}{\varphi}&=\Thprod{\displacement_h}{\grad \varphi}-\pThprod{\widehat{\displacement}_h\cdot\normal}{\varphi}, &  \quad \forall \varphi \in {W}_{h},
\end{alignat}
where the HDG numerical trace is given by $\widehat{\displacement}_h\cdot \normal=\displacement_h\cdot \normal +\tau(\phi_h -\phihat_h )$ on $\partial \Th$, for $\tau$ the stabilization operator, usually taken as a positive constant. The numerical trace unknown $\phihat_h$ is determined imposing the single-valuedness of the normal component of the numerical trace $\widehat{\displacement}_h$ on $\Faces$
\begin{alignat}{4}\label{eq:HDG-velocity-w:d}
\pThprod{\widehat{\displacement}_h\cdot \normal}{\mu}&=0, &  \quad \forall \mu\in {M}_{h}.
\end{alignat}
\end{subequations}
The system of equations \eqref{eq:HDG-velocity-w} defines the HDG semi-discrete scheme proposed in this work. An appropriate initialization of $\displacement$ in terms of the given initial data $\phi_0$ and $\velocity_0$, as prescribed in \eqref{eq:LSWEs:c} will be addressed in detail in Section~\ref{ss:initialcondition}. We now proceed to prove the energy conservation property of the semi-discrete system directly.
\begin{theorem}
Suppose that the stabilization operator $\tau>0$ is time-independent. Then, the energy (or \emph{numerical energy}) $\Hamiltonian_{2,h} = \frac{1}{2}\|\phi_h\|_{\Th}^{2} + \frac{1}{2}\|\velocity\|_{\Th,\Phi}^{2} + \frac{1}{2}\|\phi_h-\phihat_h\|_{\partial \Th, \tau}^{2}$ of the scheme \eqref{eq:HDG-velocity-w} is conserved in time, \textit{i.e.}, $\partialt{\Hamiltonian}_{2,h} = 0$, when $f=0$, $\gamma=0$ and $\vect{s}=0$.
\end{theorem}
\begin{proof}
Test \eqref{eq:HDG-velocity-w:a} with $\zb = \velocity_h$, \eqref{eq:HDG-velocity-w:b} with $\rb = \displacement_h$, and \eqref{eq:HDG-velocity-w:c} with $\varphi=\phi_h$, and add the equations together to obtain
\begin{align*}
\Thprod{\Phi\partialt{\velocity}_h}{\velocity_h}&= \Thprod{ \phi_h}{\divergence (\Phi\velocity_h)}-\pThprod{\phihat_h}{\Phi\velocity_h\cdot\normal} + \Thprod{f \Phi \velocity_{h}^\perp }{\velocity_h} \\
&= \Thprod{ \phi_h}{\divergence (\Phi\velocity_h)}-\pThprod{\phihat_h}{\Phi\velocity_h\cdot\normal} \\
& = -\Thprod{ \grad\phi_h}{\Phi\velocity_h}-\pThprod{\phihat_h - \phi_h}{\Phi\velocity_h\cdot\normal} \\
& = -\Thprod{ \grad\phi_h}{\partialt{\displacement}_h}-\pThprod{\phihat_h - \phi_h}{\partialt{\displacement}_h\cdot\normal}
\end{align*}
where we used \eqref{eq:HDG-velocity-w:b}. Now, taking time derivative in \eqref{eq:HDG-velocity-w:c} and testing with $\varphi=\phi_h$ we obtain
\[
\Thprod{\partialt{\phi}_h}{\phi_h}=\Thprod{\partialt{\displacement}_h}{\grad \phi_h}-\pThprod{\partialt{\widehat{\displacement}}_h\cdot\normal}{\phi_h} = 
\Thprod{\partialt{\displacement}_h}{\grad \phi_h}-\pThprod{\partialt{\widehat{\displacement}}_h\cdot\normal}{\phi_h-\phihat_h}
\]
where the last line results from taking time derivative of \eqref{eq:HDG-velocity-w:d} and testing with $\mu = \phihat_h$. Therefore, 
\[
\Thprod{\Phi\partialt{\velocity}_h}{\velocity_h} +\Thprod{\partialt{\phi}_h}{\phi_h} =
 - \pThprod{(\partialt{\widehat{\displacement}}_h-\partialt{\displacement}_h)\cdot\normal}{\phi_h-\phihat_h}.
\]
Thus, the result follows after using the definition of the flux $\partialt{\widehat{\displacement}}_h\cdot \normal- \partialt{\displacement}_h\cdot \normal= \tau(\partialt{\phi}_h -\partialt{\phihat}_h )$.
\end{proof}
\subsection{Initial condition for the $(\velocity-\displacement)$ formulation.} \label{ss:initialcondition}
In symplectic Hamiltonian finite element methods, the initialization plays a crucial role, as the accuracy of the semi-discrete method depends critically on it. For instance, HDG methods applied to the scalar wave equation, both in its second-order form \cite{CockburnFuHungriaSanchezSayas2018} and as a first-order system \cite{SanchezCiucaNguyenPeraireCockburn2017}, employ an initialization based on the approximation of a steady-state problem. This strategy has been shown to be essential in achieving optimal convergence. These initialization procedures typically involve computing compatible initial approximations of both the numerical trace and the volumetric variables. Treating them as independent quantities can lead to suboptimal accuracy, as demonstrated in \cite{CockburnFuHungriaSanchezSayas2018}.
In our setting, we seek initial approximations of the auxiliary variable $\displacement_h(0) \in \vect{V}_h$, the geopotential height $\phi_h(0) \in W_h$, and its numerical trace $\phihat_h(0) \in M_h$ that satisfy the steady-state equations \eqref{eq:HDG-velocity-w:c} and \eqref{eq:HDG-velocity-w:d}, given the initial data $\phi_0$. To this end, we consider the HDG approximation of the following steady-state problem: Given $\fb \in \vect{L}^2(\Omega)$, find a vector-valued function $\displacement$ such that
\begin{subequations}\label{eq:hodge-laplacian}
\begin{alignat}{4}
\label{eq:hodge-laplacian:a}
\curl\, \rot\, \displacement - \grad(\divergence \displacement) &= \fb, \quad & \text{in } \Omega, \\
\label{eq:hodge-laplacian:b}
\displacement \cdot \normal &= 0, & \quad \text{on } \Gamma, \\
\label{eq:hodge-laplacian:c}
\rot\, \displacement &= 0, & \quad \text{on } \Gamma.
\end{alignat}
\end{subequations}
This problem corresponds to the well-known two-dimensional vector-Laplacian equation, studied in  \cite{AwanouGuzmanStern2020, ArnoldFalkGopalakrishnan2012}, with magnetic boundary conditions. The connection between the solution of this problem and the initialization of the auxiliary variable becomes evident when noting that, if $\rot\, \displacement = 0$ and $\fb = \grad \phi_0$, then
\[
- \grad(\divergence \displacement) = \grad \phi_0,
\]
which implies that the auxiliary variable $\displacement$ serves as a potential function whose divergence recovers the initial geopotential gradient.
We now define an HDG approximation of the two-dimensional vector Laplacian equation \eqref{eq:hodge-laplacian}. To this end, we introduce the scalar variables $\phi = -\divergence \displacement$ and $\sigma = \rot\, \displacement$. We seek approximations $(\displacement_h, \phi_h, \sigma_h) \in \vect{V}_h \times W_h \times W_h$ to $(\displacement, \phi, \sigma)$ in the mesh $\Th$, along with the trace unknowns $(\phihat_h, \widecheck{\displacement}_h) \in M_h \times \Nhperp$ that approximate $\phi$ and the tangential component of $\displacement$ on the mesh skeleton $\Faces$, such that
\begin{subequations}\label{eq:HDG-hodge-laplacian}
\begin{alignat}{4}
\label{eq:HDG-hodge-laplacian:a}
\Thprod{\sigma_h}{\chi}-\Thprod{\displacement_h}{\curl \chi} + \pThprod{\widecheck{\displacement}_h\cdot \normal^{\perp}}{\chi}& =0, & \quad \forall \chi \in {W}_{h}, \\ 
\label{eq:HDG-hodge-laplacian:b}
\Thprod{\phi_h}{\varphi}-\Thprod{\displacement_h}{\grad\varphi}+\pThprod{\widehat{\displacement}_h\cdot \normal}{\varphi
} & =0, & \quad \forall  \varphi \in {W}_{h},
\\ 
\label{eq:HDG-hodge-laplacian:c}
\Thprod{\sigma_h}{\rot \, \zb}+
\pThprod{\sigmacheck_h}{\zb \cdot \normal^{\perp}} - 
\Thprod{\phi_h}{\divergence \zb} +
\pThprod{\phihat_h}{ \zb \cdot \normal } 
& =\Thprod{\fb}{ \zb}, &  \quad \forall \boldsymbol{ z} \in \velocity_h, 
\end{alignat}
where the numerical fluxes are given by
\begin{alignat}{4} 
\label{eq:HDG-hodge-laplacian-d} 
\sigmacheck_h=\sigma_h+\alpha^{-1}(\displacement_h\cdot\normal^{\perp}-\widecheck{\displacement}_h \cdot\normal^{\perp}) & \quad \text{ on } \partial \Th, \\
\label{eq:HDG-hodge-laplacian-e}
\widehat{\displacement}_h\cdot \normal = \displacement_h\cdot\normal+\tau (\phi_h -\phihat_h)& \quad \text{ on } \partial \Th,  
\end{alignat}
for stabilization parameters $\alpha>0$ and $\tau >0$. Imposing the weak continuity of $\sigmacheck_h$ and $\widehat{\displacement}_h\cdot \normal$, completes the system,
\begin{alignat}{4}     
\label{eq:HDG-hodge-laplacian-f}
\pThprod{\sigmacheck_h}{\mub \cdot \normal^\perp}&=0,&  &\quad \forall \mub \in \Nhperp, \\ 
\label{eq:HDG-hodge-laplacian-g}
\pThprod{\widehat{\displacement}_h \cdot \normal}{\mu} &=0,& & \quad \forall {\mu}\in {M}_{h}. 
\end{alignat}
\end{subequations}
Next, we will prove the existence and uniqueness of the HDG approximation \eqref{eq:HDG-hodge-laplacian}.
\begin{theorem}
Assume that $\alpha, \tau>0$. Then, the HDG system \eqref{eq:HDG-hodge-laplacian} is well-posed.
\end{theorem}
\begin{proof}
Assume that $\fb=0$. Testing with $\chi=\sigma_h$, $\varphi=\phi_h$ in \eqref{eq:HDG-hodge-laplacian:a} and \eqref{eq:HDG-hodge-laplacian:b} and replacing the definition \eqref{eq:HDG-hodge-laplacian-e} of $\widehat{\displacement}\cdot \normal$, and integrating by parts, we obtain
\begin{alignat*}{4}
\Thprod{\sigma_h}{\sigma_h}-\Thprod{\rot\,\displacement_h}{\sigma_h}-\pThprod{\displacement_h \cdot \normal^{\perp}}{\sigma_h} + \pThprod{\widecheck{\displacement}_h\cdot\normal^{\perp}}{\sigma_h}=0, & \\
\Thprod{\phi_h}{\phi_h}+\Thprod{\divergence \displacement_h}{\phi_h}-\pThprod{\displacement_h\cdot \normal}{\phi_h} +\pThprod{\displacement_h\cdot \normal+\tau(\phi_h-\phihat_h)}{\phi_h} =0.&
\end{alignat*}
It follows that  
\begin{alignat*}{4}
\Thprod{\rot \,\displacement_h}{\sigma_h} &=& \Thprod{\sigma_h}{\sigma_h} - \pThprod{(\displacement_h-\widecheck{\displacement}_h)\cdot\normal^{\perp}}{\sigma_h},  \\
\Thprod{\divergence \displacement_h}{\phi_h} &=&- \Thprod{\phi_h}{\phi_h}- \pThprod{\tau(\phi_h-\phihat_h)}{\phi_h}.
\end{alignat*}
Now, taking $\zb=\displacement_h$ in \eqref{eq:HDG-hodge-laplacian:c} and using the last two equations, we obtain 
\begin{alignat*}{4}
\Thprod{\sigma_h}{\sigma_h}  - \pThprod{(\displacement_h-\widecheck{\displacement}_h)\cdot\normal^{\perp}}{\sigma_h}+
\pThprod{\sigmacheck_h}{\displacement_h \cdot\normal^{\perp} } + 
\Thprod{\phi_h}{\phi_h} & \\ +
\pThprod{\tau(\phi_h-\phihat_h)}{\phi_h} + 
\pThprod{ \phihat_h}{\displacement_h \cdot \normal } & =0.
\end{alignat*}
The last expression can be rewritten as
\begin{alignat*}{4}
\Thprod{\sigma_h}{\sigma_h} + 
\Thprod{\phi_h}{\phi_h} 
+\pThprod{(\displacement_h -\widecheck{\displacement}_h)\cdot \normal^\perp}{\alpha^{-1}(\displacement_h -\widecheck{\displacement}_h) \cdot \normal^\perp} &
\\ +\pThprod{\tau (\phi_h -\phihat_h)}{(\phi_h -\phihat_h)} +  
\pThprod{\sigmacheck_h}{\widecheck{\displacement}_h \cdot \normal^\perp}+
\pThprod{\widehat{\displacement}_h \cdot \normal}{\phihat_h}
 &=0.
\end{alignat*}
Using the continuity conditions \eqref{eq:HDG-hodge-laplacian-f} and \eqref{eq:HDG-hodge-laplacian-g}, we observe that the last two terms of the left-hand side are zero. Thus, we obtain
 \begin{alignat*}{4}
\Thprod{\sigma_h}{\sigma_h} + 
\Thprod{\phi_h}{\phi_h} 
+\pThprod{\alpha^{-1}(\displacement_h -\widecheck{\displacement}_h)\cdot \normal^\perp}{(\displacement_h -\widecheck{\displacement}_h) \cdot \normal^\perp}& \\
+\pThprod{\tau (\phi_h -\phihat_h)}{(\phi_h -\phihat_h)} &=0.
\end{alignat*}
Since we have the sum of four non-negative terms equal to zero, we can deduce that each term is zero. Therefore, it follows that $ \sigma_h=0, \ \phi_h=0, \  \displacement_h \cdot \normal^\perp=\widecheck{\displacement}_h \cdot \normal^\perp$ and $\phihat_h=0$ on $\partial \Th$. Replacing these zero values in the continuity equations gives that
 \begin{alignat*}{4}
\pThprod{\displacement_h \cdot \normal}{\mu} = 0,
\end{alignat*}
for all $\mu \in M_h$.
Thus,  replacing $\sigma_h=0$ and $\displacement_h \cdot \normal^\perp=\widecheck{\displacement}_h \cdot \normal^\perp$ in \eqref{eq:HDG-hodge-laplacian:a}, we obtain
\[
\Thprod{\rot\,\displacement_h}{\chi} = 0.
\]
Analogously, replacing $\phi_h = 0 $ and $\phihat_h =0$ in \eqref{eq:HDG-hodge-laplacian:b}, we obtain that
\[
\Thprod{\divergence \displacement_h}{\varphi} = 0.
\]
Since $\rot \, \boldsymbol{\mathcal{P}}_k(K)\subseteq \mathcal{P}_k(K)$ and $\divergence \boldsymbol{\mathcal{P}}_k(K) \subseteq \mathcal{P}_k(K)$ then we conclude that $\rot \, \displacement_h=0$ and $\divergence \displacement_h=0$.
Using the decomposition, for $\displacement_h\in \ring{H}(\divergence)\bigcap H(\rot)$, \cite{girauld}
\[
\displacement_h = \grad{\psi} + \curl\, q,
\]
where $\psi\in H^1(\Omega) \bigcap L_0^2(\Omega)$ and $q\in H^1(\Omega)$ are solutions of the problems
\begin{equation*}
\begin{minipage}{0.45\textwidth}
\begin{alignat*}{2}
    -\Delta \psi &= \divergence \displacement_h = 0 \quad &&\text{in } \Omega, \\
    \frac{\partial \psi}{\partial \normal} &= 0 \quad &&\text{on } \Gamma,
\end{alignat*}
\end{minipage}
\hfill
\text{and}
\hfill
\begin{minipage}{0.45\textwidth}
\begin{alignat*}{2}
    -\Delta q &= \rot \displacement_h = 0 \quad &&\text{in } \Omega, \\
    q &= 0 \quad &&\text{on } \Gamma.
\end{alignat*}
\end{minipage}
\end{equation*}
Then, it follows that $\displacement_h = \grad{0} + \rot\, 0 = 0$.
Therefore, we have shown that the solution operator is injective, and given that we are working in finite-dimensional spaces, we deduce that the operator is bijective, and consequently, the problem is well-posed.
\end{proof}
\subsection{HDG scheme for the $(\phi-\velocity)$ formulation.}\label{ss:HDGphiv}
We conclude this section by revisiting the system \eqref{eq:LSWEs} and the hybridizable discontinuous Galerkin method introduced in \cite{Bui-Thanh2016}, with the aim of comparing it to the scheme presented in \eqref{eq:HDG-velocity-w}. Consider the formulation \eqref{eq:LSWEs}. The corresponding HDG scheme is stated as follows: Find approximations to the geopotential height and velocity, $(\phi_h, \velocity_h) \in W_h \times \vect{V}_h$, along with a numerical trace approximation $\phihat_h \in M_h$, such that
\begin{subequations}\label{eq:HDG-phi-velocity}
\begin{alignat}{4}
\label{eq:HDG-phi-velocity:a}
\Thprod{\partialt{\phi}_h}{\varphi} & = 
\Thprod{\Phi \velocity_h}{\grad \varphi}-
\pThprod{\Phi \widehat{\velocity}_h\cdot \normal}{\varphi},&  
\quad \forall\varphi\in {W}_{h},\\  
\label{eq:HDG-phi-velocity:b}
\Thprod{\Phi\partialt{\velocity}_h}{\zb}&=
\Thprod{\phi_h}{\divergence (\Phi\zb)}-
\pThprod{\phihat_h}{\Phi\zb \cdot \normal} +
\Thprod{f \Phi\velocity_h^{\perp}}{\zb}-\Thprod{\vect{s} }{\zb}, &   
\quad \forall \zb \in \vect{V}_{h},
\end{alignat}
where  the numerical flux is given by $\Phi\widehat{\velocity}_h\cdot \normal=\Phi\velocity_h\cdot \normal +\tau(\phi_h -\phihat_h )$ on $\partial \Th $. The continuity equation and boundary condition sets the continuity of the flux
\begin{alignat}{4}
\label{eq:HDG-phi-velocity:c}
\pThprod{\widehat{\velocity}_h\cdot \normal}{\mu} &=0, &  \quad \forall {\mu}\in {M}_{h}.
\end{alignat}
\end{subequations}
This scheme corresponds to the HDG-II method of \cite{Bui-Thanh2016}, choosing the stabilization parameter $\tau = \sqrt{\Phi}$. The following theorem, originally stated in \cite{Bui-Thanh2016}, is recalled here to illustrate that the energy of the system is dissipative.
\begin{theorem}
The energy $\mathcal E_h = \frac{1}{2}\|\phi_h\|_{\Th}^{2} + \frac{1}{2}\|\velocity\|_{\Th,\Phi}^{2}$, of the scheme \eqref{eq:HDG-phi-velocity} is dissipative in time, \textit{i.e.},  $\partialt{\mathcal E}_h\leq 0$.
\end{theorem}
\begin{proof}
Take $\varphi=\phi_h$  in \eqref{eq:HDG-phi-velocity:a} and  $\zb = \velocity_h$ in $\eqref{eq:HDG-phi-velocity:b}$, to obtain
\begin{align*}
\Thprod{\partialt{\phi}_h}{\phi_h} + \Thprod{\Phi\partialt{\velocity}_h}{\velocity_h} & = 
\Thprod{\Phi \velocity_h}{\grad \phi_h}-
\pThprod{\Phi \widehat{\velocity}_h\cdot \normal}{\phi_h}
+
\Thprod{\phi_h}{\divergence (\Phi\velocity_h)}\\
&\quad  -
\pThprod{\phihat_h}{\Phi\velocity_h \cdot \normal} +
\Thprod{f \Phi\velocity_h^{\perp}}{\velocity_h}\\
& = 
-\pThprod{(\Phi \widehat{\velocity}_h - \Phi \velocity_h)\cdot \normal}{\phi_h}
-\pThprod{\phihat_h}{\Phi\velocity_h \cdot \normal} 
\\
&= 
\pThprod{\phihat_h-\phi_h}{\Phi(\widehat{\velocity}_h - \velocity_h)\cdot\normal} 
\\
&= -\pThprod{\phi_h-\phihat_h}{\tau(\phi_h-\phihat_h} 
\end{align*}
where we used integration by parts and \eqref{eq:HDG-phi-velocity:c} with $\mu = \phihat_h$. This implies that
\[
\partialt{\mathcal E}_h =-\|\phi_h-\phihat_h\|_{\partial \Th, \tau}^{2} \leq 0. 
\]
\end{proof}
\section{The Hamiltonian structure of the  semidiscrete HDG method}\label{sec:section4}
This section analyzes the conservation properties of the HDG scheme \eqref{eq:HDG-velocity-w}, based on the $(\velocity-\displacement)$ formulation. In particular, we show that the system has a Hamiltonian structure.
\subsection{Hamiltonian structure of $(\velocity-\displacement)$ formulation}\label{ss:Hamiltonianstructure-u-w}
In this section, we aim to prove that the HDG scheme \eqref{eq:HDG-velocity-w}, based on the formulation involving the velocity $\velocity$ and the auxiliary variable $\displacement$, defines a Hamiltonian system in the sense of \eqref{eq:HPDE}.  To this end, we introduce the triplet $(\phasespace_{2,h}, \Hamiltonian_{2,h}, \Pbrack{\cdot}{\cdot}_{2,h})$, associated with the HDG method, which characterizes the discrete Hamiltonian system. Let the discrete phase-space be defined as $\phasespace_{2,h} = \vect{V}_h \times \vect{V}_h$. The discrete Hamiltonian functional $\Hamiltonian_{2,h}$, defined for functions $\velocity_h$ and $\displacement_h$, is given by
\begin{alignat}{4}
\label{eq:Hamiltonian_h}
\Hamiltonian_{2,h}[\velocity_h, \displacement_h]=
\frac{1}{2}\Thprod{\phi_h}{\phi_h}+
\frac{1}{2}\Thprod{\Phi\,\velocity_h}{\velocity_h}+
\frac{1}{2}\pThprod{\tau( \phi_{h}-\phihat_{h})}{\phi_{h}-\phihat_{h}}
\end{alignat}
where $\phi_h$ and $\phihat_h$ are given in terms of $\displacement_h$ by equations $\eqref{eq:HDG-velocity-w:c}$ and $\eqref{eq:HDG-velocity-w:d}$. The discrete Poisson bracket $\Pbrack{\cdot}{\cdot}_{2,h}$ is defined analogously to the continuous one for functionals restricted to functions in the discrete phase-space $F=F[\velocity_h,\displacement_h]$ and $G=G[\velocity_h,\displacement_h]$,
\begin{alignat}{4}
\label{eq:Poissobbracket_h}
\Pbrack{F}{G}_{2,h} =  
\Thprod{\fder{F}{\displacement_h}}{\fder{G}{\velocity_h}}-    
\Thprod{\fder{F}{\velocity_h}}{\fder{G}{\displacement_h}}+
\Thprod{f\Phi^{-1}\left(\fder{G}{\velocity_h}\right)^{\perp}}{\fder{F}{\velocity_h}}.
\end{alignat}
The discrete Hamiltonian functional $\Hamiltonian_{2,h}$ defined in \eqref{eq:Hamiltonian_h} differs from the continuous Hamiltonian \eqref{eq:Hamiltonian} by an additional quadratic term involving the numerical trace $\phihat_h$. We interpret this as the \emph{numerical energy} associated with the HDG scheme \eqref{eq:HDG-velocity-w}. We now prove the main result of the paper, the Hamiltonian structure of system \eqref{eq:HDG-velocity-w} for $\vect{s}=0$.
\begin{theorem}\label{teo:HamiltoniansystemHDG}
The HDG method defined by the time evolution equations \eqref{eq:HDG-velocity-w:a} -\eqref{eq:HDG-velocity-w:b} with $\vect{s}=0$ and the steady-state equations \eqref{eq:HDG-velocity-w:c}-\eqref{eq:HDG-velocity-w:d} is a Hamiltonian dynamical system with $(\phasespace_{2,h},\Hamiltonian_{2,h}, \Pbrack{\cdot}{\cdot}_{2,h})$.
\end{theorem}
\begin{proof}
We begin computing the variation of the Hamiltonian functional $\Hamiltonian_{2,h}$
\[
\delta \Hamiltonian_{2,h} =  
\Thprod{\phi_h}{\delta \phi_h}+
\Thprod{\Phi \velocity_h}{\delta \velocity_h}+ 
\pThprod{\tau(\phi_{h}-\phihat_{h})}{\delta( \phi_{h}-\phihat_{h})}. 
\]
We rewrite the first and third terms above in terms of variations of $\vect{w}_h$. Considering \eqref{eq:HDG-velocity-w:c}, we observe that if we take the variation and test with $\varphi = \phi_h$ we obtain
\[
\Thprod{\delta \phi_h}{\phi_h}=
\Thprod{\delta \displacement_h}{\grad \phi_h}- 
\pThprod{\delta \widehat{\displacement}_{h} \cdot \normal}{\phi_h}.
\]
Then, it follows that
\begin{align*}
\delta \Hamiltonian_{2,h} & =  
\Thprod{\delta \displacement_h}{\grad \phi_h}- 
\pThprod{\delta \widehat{\displacement}_{h} \cdot \normal}{\phi_h}+
\Thprod{\Phi \velocity_h}{\delta \velocity_h}+ 
\pThprod{\tau(\phi_{h}-\phihat_{h})}{\delta( \phi_{h}-\phihat_{h})}.
\end{align*}
Similarly, taking variation in \eqref{eq:HDG-velocity-w:d} and testing with $\mu=\phihat_h$, and then adding this equation to the expression above, we obtain
\begin{align*}
\delta \Hamiltonian_{2,h} & = 
\Thprod{\delta \displacement_h}{\grad \phi_h}- 
\pThprod{\delta \widehat{\displacement}_{h} \cdot \normal}{\phi_h-\phihat_h}+
\Thprod{\Phi \velocity_h}{\delta \velocity_h}+ 
\pThprod{\tau(\phi_{h}-\phihat_{h})}{\delta( \phi_{h}-\phihat_{h})} \\
& =
\Thprod{\delta \displacement_h}{\grad \phi_h}+
\Thprod{\Phi \velocity_h}{\delta \velocity_h} 
-\pThprod{\delta(\widehat{\displacement}_h \cdot\normal- \tau(\phi_{h}-\phihat_{h}))}{( \phi_{h}-\phihat_{h})}\\
& =
\Thprod{\delta \displacement_h}{\grad \phi_h}+
\Thprod{\Phi \velocity_h}{\delta \velocity_h} 
-\pThprod{\delta\displacement_h \cdot\normal }{( \phi_{h}-\phihat_{h})}
\end{align*}
where we used the definition of the numerical flux $\widehat{\displacement}_h\cdot \normal$.

To prove the Hamiltonian structure of the HDG methods, we define the so-called \textit{coordinate functional} associated with the time evolution of $(\velocity_h,\displacement_h)$
\[
F_{h}[\velocity_h, \displacement_h] = \Thprod{\Phi \velocity_h}{\zb},+ \Thprod{\displacement_h}{\rb}, \quad \text{ for }(\zb,\rb)\in \vect{V}_h \times\vect{V}_h.
\]
Note that
\[
\fder{F_{h}}{\velocity_h} = \Phi\zb, \quad  \fder{F_{h}}{\displacement_h} = \rb.
\]
Now, we evaluate the discrete Poisson bracket to observe the equivalence with the HDG method
\begin{align*}
\Thprod{\Phi\partialt{\velocity}_h}{\zb}+\Thprod{\partialt{\displacement}_h}{\rb} = 
\partialt{F}_{h} 
= 
\Pbrack{F_{h}}{\Hamiltonian_{2,h}}_{2,h}
&= -\Thprod{\grad \phi_h}{\Phi\zb} + 
\pThprod{(\phi_h-\phihat_h)}{\Phi\zb \cdot\normal} \\
&\quad  +  \Thprod{f \Phi^{-1}(\Phi \velocity_h)^\perp}{\Phi \zb} \\
& =
\Thprod{\phi_h}{\divergence (\Phi \zb)} 
-\pThprod{\phihat_h}{\Phi\zb \cdot\normal}\\
&\quad + 
\Thprod{\Phi f \velocity_h^\perp}{\zb}+\Thprod{\Phi \velocity_h}{\rb}, 
\end{align*}
which proves that system \eqref{eq:HDG-velocity-w:a}-\eqref{eq:HDG-velocity-w:b} is a Hamiltonian system. This completes the proof.
\end{proof}
Using the Hamiltonian structure, precisely the Poisson bracket of the formulation \eqref{eq:HDG-velocity-w}, we can prove discrete versions of conservation laws, particularly energy conservation. We state this in the following lemma.
\begin{lemma}
The HDG approximations in \eqref{eq:HDG-velocity-w} satisfies the following 
\begin{alignat*}{4}
(\text{mass})& \quad \Thprod{\phi_h}{1}&=&0, \\ 
(\text{energy})& \quad \partialt{\Hamiltonian}_{2,h}&=&0.
\end{alignat*}
\end{lemma}
\begin{proof}
Taking $\varphi=1$ in the equation \eqref{eq:HDG-velocity-w:c}, we obtain
\begin{align*}
\Thprod{\phi_h}{1}&=\Thprod{\displacement_h}{\grad 1 }-\pThprod{\widehat{\displacement}_h\cdot\normal}{1}=0,
\end{align*}
as a consequence of equation \eqref{eq:HDG-velocity-w:d}.

Energy conservation is achieved through the anti-symmetry property of the Poisson bracket, that is
\[
\partialt{\Hamiltonian}_{2,h} = \Pbrack{\Hamiltonian_{2,h}}{\Hamiltonian_{2,h}}_{2,h} = 0.
\]

\end{proof}
\section{Fully discrete symplectic HDG schemes}\label{sec:section5}
This section presents a matrix formulation of the HDG method along with an alternative proof of its Hamiltonian structure. We then introduce fully discrete schemes for the $(\velocity-\displacement)$ formulation \eqref{eq:LSWEs2} by coupling the HDG method with high-order symplectic time integrators. Two classes are considered: sDIRK methods, which exactly conserve energy for quadratic Hamiltonians, and sEPRK methods, suitable for separable Hamiltonians and known for their non-drifting energy behavior \cite{sanzserna,McLachlan94}. 
\subsection{Finite dimensional Hamiltonian system}\label{ss:finite-dimensional-Hamiltonian}
We follow the approach described in \cite{CockburnDuSanchez2023} and present the matrix representation of the semi-discrete HDG scheme \eqref{eq:HDG-velocity-w}. To this end, let $\left\{\varphi_{i}\right\}$, $\left\{\zb_{k}\right\}$, and $\left\{\eta_{m}\right\}$ denote orthonormal bases for the finite element spaces $W_h$, $\vect{V}_h$, and $M_h$, respectively, with indices ranging over the dimensions of the corresponding spaces. We define the following matrices
\begin{alignat*}{4} 
\mathrm{A}_{kl}&:=\Thprod{\zb_{l}^{\perp}}{\zb_{k}},&  \quad
\mathrm{B}_{ki}&:=\Thprod{\varphi_i}{\divergence\zb_{k}},& \quad 
\mathrm{C}_{km}&:=\pThprod{\eta_m}{\zb_{k} \cdot \normal}, 
\\
\mathrm{S}_{ij}&:=\pThprod{\tau \varrho_j}{\varrho_{i}},& \quad
\mathrm{E}_{im}&:=\pThprod{\tau \eta_m}{\varrho_{i}}, &\quad 
\mathrm{G}_{mn}&:=\pThprod{\tau \eta_n}{\eta_{m}}.
\end{alignat*}
Subsequently, the  unknowns $(\velocity_h,\displacement_h, \phi_h, \phihat_h)$ of system \eqref{eq:HDG-velocity-w} are represented in terms of the coefficient vectors $(\mathsf{u}, \mathsf{w}, \mathsf{p}, \widehat{\mathsf{p}})$ as follows
\begin{alignat*}{4}
\velocity_{h}(x, t)&=\sum_{k} \mathsf{u}_{k}(t) \zb_{k}(x), 
& \quad \displacement_{h}(x, t)&=\sum_{k} \mathsf{w}_{k}(t) \zb_{i}(x), \\ 
\phi_{h}(x, t)&=\sum_{i} \mathsf{p}_{i}(t) \varrho_{i}(x),
 & \quad  \phihat_{h}(x, t)&=\sum_{m} \widehat{\mathsf{p}}_{m}(t) \eta_{m}(x) .
\end{alignat*}
Thus, we rewrite the HDG scheme \eqref{eq:HDG-velocity-w} as the following system of ordinary differential equations: find  the coefficient vectors $(\mathsf{u}, \mathsf{w}, \mathsf{p}, \widehat{\mathsf{p}})$ such that
\begin{subequations}\label{eq:ODEHDG-velocity-w}
\begin{alignat}{4}
\label{eq:ODEHDG-velocity-w:a}
\partialt{\mathsf{w}}&= \Phi\mathsf{u},\\ 
\label{eq:ODEHDG-velocity-w:b}
\partialt{\mathsf{u}}&= \mathrm{B} \mathsf{p}- \mathrm{C} \widehat{\mathsf{p}} + \mathrm{A} \mathsf{u}  \\
\label{eq:ODEHDG-velocity-w:c}
0 & =\mathsf{p} + \mathrm{B}^\top \mathsf{w} + \mathrm{S}\mathsf{p} -  \mathrm{E} \widehat{\mathsf{p}} \\
\label{eq:ODEHDG-velocity-w:d}
0 & = \mathrm{C}^{\top} \mathsf{w} + \mathrm{E}^{\top} \mathsf{p} - \mathrm{G} \widehat{\mathsf{p}}
\end{alignat}
\end{subequations}
The Hamiltonian functional \eqref{eq:Hamiltonian_h} is written in terms of the coefficients $\mathsf{u}$ and $\mathsf{v}$ as follows
\begin{alignat}{4}\label{eq:ODEHamiltonian_h}
\mathsf{H}[\mathsf{u}, \mathsf{w}] = 
\frac{1}{2} \mathsf{p}^\top\mathsf{p} + 
\frac{1}{2}\Phi \mathsf{u}^\top\mathsf{u} + 
\frac{1}{2} \begin{bmatrix} \mathsf{p}^\top, \widehat{\mathsf{p}}^\top \end{bmatrix}\begin{bmatrix} \mathrm{S} & -\mathrm{E} \\ -\mathrm{E}^\top & \mathrm{G} \end{bmatrix} \begin{bmatrix}  \mathsf{p} \\ \widehat{\mathsf{p}}\end{bmatrix}
\end{alignat}
We now prove that \eqref{eq:ODEHDG-velocity-w} is a Poisson system, or generalized Hamiltonian system, according to \cite{CockburnDuSanchez2023}.
\begin{theorem}\label{thm:ODEHDG-Hamiltonian}
The system of ODEs \eqref{eq:ODEHDG-velocity-w} is a Poisson system, \textit{i.e.}, is equivalent to $\partialt{\yb} = \mathrm{J} \nabla \mathsf{H}$, for $\yb = [\mathsf{u}, \mathsf{w}]$, with Hamiltonian function $\mathsf{H} = \mathsf{H}[\mathsf{u}, \mathsf{w}]$ defined in \eqref{eq:ODEHamiltonian_h} and structure matrix
\[
\mathrm{J} = 
\begin{bmatrix}
    f\Phi^{-1}\mathrm{A} & -\mathrm{I} \\
    \mathrm{I} & 0
\end{bmatrix}.
\]
\end{theorem}
\begin{proof}
We compute the derivatives of the Hamiltonian function. First, we observe that
\(
\partial \mathsf{H}/\partial \mathsf{u} = \Phi \mathsf{u}.
\)
To compute the partial derivative with respect to $\displacement$ we rearrange the term in the Hamiltonian as
\[
\frac{\partial \mathsf{H}}{\partial \mathsf{w}} = \frac{\partial}{\partial \mathsf{w}}
\left(
\frac{1}{2} \begin{bmatrix} \mathsf{p}^\top, \widehat{\mathsf{p}}^\top \end{bmatrix}\begin{bmatrix} \mathrm{I} + \mathrm{S} & -\mathrm{E} \\ -\mathrm{E}^\top & \mathrm{G} \end{bmatrix} \begin{bmatrix}  \mathsf{p} \\ \widehat{\mathsf{p}}\end{bmatrix}
\right)
=
\left(
\begin{bmatrix} \mathsf{p}^\top, \widehat{\mathsf{p}}^\top \end{bmatrix}\begin{bmatrix} \mathrm{I} + \mathrm{S} & -\mathrm{E} \\ -\mathrm{E}^\top & \mathrm{G} \end{bmatrix} \begin{bmatrix}  \partial \mathsf{p}/\partial \mathsf{w} \\ \partial \widehat{\mathsf{p}}/ \partial \mathsf{w}\end{bmatrix}
\right)^\top.
\]
Using \eqref{eq:ODEHDG-velocity-w:c}-\eqref{eq:ODEHDG-velocity-w:d} we arrive to the system
\[
\begin{bmatrix}
\mathrm{I} + \mathrm{S} & -\mathrm{E} \\
-\mathrm{E}^\top & \mathrm{G}
\end{bmatrix}
\begin{bmatrix}
\mathsf{p} \\
\widehat{\mathsf{p}}
\end{bmatrix}
=
\begin{bmatrix}
-\mathrm{B}^\top \\
\mathrm{C}^\top
\end{bmatrix}
\mathsf{w}.
\]
Differentiating with respect to $\mathsf{w}$ the system above, we can conclude  that
\[
\frac{\partial \mathsf{H}}{\partial \mathsf{w}} = 
\left(
\begin{bmatrix} \mathsf{p}^\top, \widehat{\mathsf{p}}^\top \end{bmatrix}
\begin{bmatrix}
-\mathrm{B}^\top \\
\mathrm{C}^\top
\end{bmatrix}
\right)^\top
= -\mathrm{B} \mathsf{p} + \mathrm{C} \widehat{\mathsf{p}}.
\]
Then, the theorem statement follows from the definition of the anti-symmetric structure operator $\mathrm{J}$ and the partial derivatives of $\mathsf{H}$ which turn equivalent to the dynamical system \eqref{eq:ODEHDG-velocity-w:a}-\eqref{eq:ODEHDG-velocity-w:b}. This completes the proof.
\end{proof}
\subsection{Symplectic Diagonally Implicit Runge-Kutta Methods}\label{ss:sDIRK}
In this section, we describe the time discretization of HDG scheme \eqref{eq:HDG-velocity-w} using high-order symplectic diagonally implicit Runge–Kutta methods (sDIRK). To this end, we consider a general initial value problem of the form $\partialt{ \yb}(t) = \fb(t, \yb(t))$, with initial condition $\yb(0) = \yb_0$. Given a fixed time step $\Delta t$, an sDIRK scheme computes an approximation $\yb^{n+1} \approx \yb(t^{n+1})$ at time $t^{n+1} = (n+1)\Delta t$ using the update formula
\[
\yb^{n+1}=\yb^{n}+\Delta t \sum_{i=1}^{s} b_{i} k_{i}, \quad \yb^{0} = \yb_0,
\]
where $k_i = \fb\left(t^{n,i}, \yb^{n,i}\right)$ and $\yb^{n,i} = \yb^{n} + \Delta t \sum_{j=1}^{i} a_{ij} k_j$, with intermediate times $t^{n,i} = t^n + c_i \Delta t$ for $1 \leq i \leq s$. The coefficients $A = (a_{ij})$, together with the vectors $b = (b_i)$ and $c = (c_i)$, define the Butcher tableau associated with the Runge–Kutta method; see, for example,~\cite{sanzserna}. In diagonally implicit Runge–Kutta (DIRK) schemes, the matrix $A$ is lower triangular, that is, $a_{ij} = 0$ for $j > i$. Moreover, symplectic Runge–Kutta methods are characterized by the following condition on the coefficients
\begin{alignat}{4}\label{eq:RKsymplectic}
b_{i} a_{i j}+b_{j} a_{j i}-b_{i} b_{j}=0, \quad 1 \leq i, j \leq s.
\end{alignat}
The HDG scheme \eqref{eq:ODEHDG-velocity-w} corresponds to a linear initial value problem, of the form \(\partialt{\yb}  = \mathbb{M} \yb \), for $\yb = [\mathsf{w}, \mathsf{u} ]^{\top}$, and where 
\[
\mathbb M = 
\begin{bmatrix}
0 & \mathrm{I} \\
-\mathrm{K} & -\mathrm{A}
\end{bmatrix}
, \quad \text{where}\quad
\mathrm{K} = 
\begin{bmatrix}
-\mathrm{B} & \mathrm{C}
\end{bmatrix}
\begin{bmatrix}
\mathrm{I} + \mathrm{S} & -\mathrm{E} \\
-\mathrm{E}^\top & \mathrm{G}
\end{bmatrix}^{-1}
\begin{bmatrix}
-\mathrm{B}^\top \\ \mathrm{C}^\top
\end{bmatrix}.
\]
Thus, at each stage $1\leq i\leq s$,  the DIRK scheme solves
\begin{subequations}\label{eq:sDIRK}
\begin{align}\label{eq:sDIRK:a}
\yb^{n,i} -\Delta t a_{ii} \mathbb M \yb^{n,i}  & = \yb^{n} + \Delta t \sum_{j=1}^{i-1}a_{ij}k_{j}, \\ \label{eq:sDIRK:b}
k_{i} &= \frac{1}{a_{ii}\Delta t} \left(\yb^{n,i}-\yb^{n} -\Delta t\sum_{j=1}^{i-1}a_{ij}k_{j} \right),
\end{align}
and the solution is then updated as \begin{align}\label{eq:sDIRK:c}
\yb^{n+1} = \yb^{n} + \Delta t \sum_{i=1}^{s} b_{i} k_{i}.
\end{align}
\end{subequations}
\subsection{Explicit Symplectic Partitioned Runge-Kutta Methods}
The symplecticity condition for Runge–Kutta methods, given by \eqref{eq:RKsymplectic}, implies that such schemes are necessarily implicit. However, for the special class of separable Hamiltonian systems such as \eqref{eq:Hamiltonian}, partitioned methods can exploit the underlying structure to construct explicit symplectic integrators. To introduce this class of methods, consider systems of the form $\partialt \yb = \fb(t, \yb, \zb)$ and $\partialt \zb = \gb(t, \yb, \zb)$. A partitioned Runge–Kutta method, see \cite{sanzserna2}, employs two sets of Runge–Kutta coefficients, $(b_i, a_{ij})$ and $(\hat{b}_i, \hat{a}_{ij})$, and is defined by the following update rule
\begin{align*}
\yb^{n+1} & = \yb^{n} + \Delta t \sum_{i=1}^{s}b_{i} k_{i},  &
k_{i} & = \fb(t^{n,i}, \yb^{n} + \Delta t \sum_{j=1}^{s} a_{ij} k_{j}, \zb^{n} + \Delta t \sum_{j=1}^{s} \hat{a}_{ij} \ell_{j}), \\
\zb^{n+1} & = \zb^{n} + \Delta t \sum_{i=1}^{s}\hat{b}_i \ell_{i}, &
\ell_{i} & = \gb ( \hat{t}^{n,i}, \yb^{n} + \Delta t \sum_{j=1}^{s} a_{ij} k_{j}, \zb^{n} + \Delta t \sum_{j=1}^{s} \hat{a}_{ij} \ell_{j}),
\end{align*}
for $t^{n,i} = t^{n} + c_{i}\Delta t$ and $\hat{t}^{n,i} = t^{n} + \hat{c}_i\Delta t$. Explicit schemes can be derived under appropriate conditions on $\fb$ and $\gb$, together with suitable choices of Butcher tableaus. For instance, if $\fb = \fb(t, \zb)$, it is possible to combine an explicit and a diagonally implicit Runge–Kutta scheme. Moreover, the resulting method is symplectic if the coefficients satisfy the following condition
\begin{alignat*}{4}
b_{i} \tilde{a}_{i j}+\tilde{b}_{j} a_{j i}-b_{i} \tilde{b}_{j} &= 0, \quad 1 \leq i, j \leq s,
\end{alignat*}
for further details see \cite{sanzserna}. In the case of the semi-discrete scheme \eqref{eq:ODEHDG-velocity-w}, it takes the form
\[
\partialt{\mathsf{w}} = \mathsf{u}, \qquad 
\partialt{\mathsf{u}} = -\mathrm{K} \mathsf{w} - \mathrm{A} \mathsf{u}.
\]
To construct an explicit scheme, we consider Butcher tableaus $(a_{ij}, b_i, c_i)$ and $(\hat{a}_{ij}, \hat{b}_i, \hat{c}_i)$ corresponding to $s$-stage diagonally implicit and explicit Runge–Kutta schemes, respectively. The fully discrete method advances the solution from time $t^n$ to $t^{n+1}$ using the following update rule
\begin{subequations}\label{eq:sESPRK}
\begin{align} \label{eq:sESPRK:a}
\mathsf{w}^{n+1} &= \mathsf{w}^{n} + \Delta t \sum_{i=1}^{s} b_{i} k_{i}, &
k_{i} &= \Phi \left(\mathsf{u}^{n} + \Delta t \sum_{j=1}^{i-1} \hat{a}_{ij} \ell_{j} \right), \\ \label{eq:sESPRK:b}
\mathsf{u}^{n+1} &= \mathsf{u}^{n} + \Delta t \sum_{i=1}^{s} \hat{b}_{i} \ell_{i}, &
\ell_{i} &= -\mathrm{K} \left(\mathsf{w}^{n} + \Delta t \sum_{j=1}^{i} a_{ij} k_{j} \right) 
            - \mathrm{A} \left(\mathsf{u}^{n} + \Delta t \sum_{j=1}^{i-1} \hat{a}_{ij} \ell_{j} \right).
\end{align}
\end{subequations}
\section{Numerical experiments}\label{sec:section6}
This section presents three numerical experiments to assess the performance of the HDG scheme \eqref{eq:HDG-velocity-w}, introduced in Section~\ref{sec:section4} for the $(\velocity-\displacement)$ formulation combined with symplectic integrators of Section~\ref {sec:section5}.  The first experiment focuses on verifying the convergence of the symplectic Hamiltonian HDG methods using uniform meshes. We begin with a numerical study of the convergence rates of the initialization described in \eqref{eq:HDG-hodge-laplacian}, and the HDG method combined with the sEPRK schemes. We use polynomials of order $k$ for the space approximations and time integrators of order $k+2$ described by \eqref{eq:sESPRK}, computing the convergence at the final time. The second experiment investigates the long-time behavior of some physical quantities of interest using the symplectic midpoint time-integrator scheme.  Finally, the third experiment explores the behavior of the physical quantities under variable bathymetry. All numerical simulations were carried out using the open source finite element library NETGEN \cite{Schoeberl1997} and NGSolve \cite{Schoeberl2014}.
\subsection*{Experiment 1: Linear standing wave}\label{ss:Experiment1}
We consider the linear standing wave problem described in \cite{Bui-Thanh2016}, and approximate a manufactured solution for the initial boundary value problem \eqref{eq:LSWEs2} on the domain $\Omega = (0,1)^2$.
The exact solution is given by
\[
\phi(x, y, t) = \cos(\pi x)\cos(\pi y)\cos(\sqrt{2} \pi t),
\]
and
\[
\displacement(x, y, t)=
\begin{bmatrix}
-\frac{1}{2\pi}\sin(\pi x)\cos(\pi y)\cos(\sqrt{2}\pi t) \\
-\frac{1}{2\pi}\cos(\pi x) \sin(\pi y)\cos(\sqrt{2}\pi t) 
\end{bmatrix}, \quad
\velocity(x, y, t)=
\begin{bmatrix}
\frac{1}{\sqrt{2}}\sin(\pi x)\cos(\pi y)\sin(\sqrt{2}\pi t)\\
\frac{1}{\sqrt{2}}\cos(\pi x) \sin(\pi y)\sin(\sqrt{2}\pi t) 
\end{bmatrix}.
\]
The parameters for this problem are $\Phi = 1$, $f = 0$, $\gamma = 0$, $\vect{s}=0$ and $g = 1$.  The computations are performed on uniform triangulations with mesh size $h = 2^{-l}$, for $l = 1, \dots, 5$. We report the $L^2$-norm errors of the approximations and the estimated order of convergence (e.o.c.), defined as
\begin{alignat*}{4}
\operatorname{error}(h) &= \max_{n} \left\| \star(t^n) - \star_h^n \right\|_{L^2(\Omega)}, \quad
\text{e.o.c.}(h) &= \frac{\log\left(\operatorname{error}(h) / \operatorname{error}(h')\right)}{\log(h / h')},
\end{alignat*}
for $\star = \phi, \displacement, \velocity, \sigma$, where $h'$ denotes the size of the next coarser mesh.

Table~\ref{tab:exp1:initialcondition} shows the convergence history for the initial condition obtained using the HDG method \eqref{eq:HDG-hodge-laplacian}, applied to the two-dimensional vector Laplacian problem \eqref{eq:hodge-laplacian} with data $\fb = \nabla \phi(x,y,0)$. The results show optimal convergence of order $k+1$ for the approximations $\sigma_h$, $\phi_h$, and $\displacement_h$. All simulations used stabilization parameters $\tau = \alpha = 1$. 
\begin{table}[htbp]
\centering
\scriptsize
\begin{tabular}{@{}lcl@{\hskip .2in}cc@{\hskip .1in}c@{\hskip .1in}cc@{\hskip .1in}c@{\hskip .1in}cc}
\toprule
&&  & \multicolumn{2}{c}{$\sigma_h$} && \multicolumn{2}{c}{$\displacement_h$} && \multicolumn{2}{c}{${\phi}_{h}$}  \\
\cmidrule(lr){4-5} \cmidrule(lr){7-8} \cmidrule(lr){10-11} 
$k$
& $h$     &&  error   &e.o.c.&&  error   &e.o.c.&&  error   & e.o.c. \\
\midrule
\multirow{5}{*}{0}
&5.00e-01 && 3.62e-02 & -    && 3.06e-01 & -    && 2.77e-01 & - \\
&2.50e-01 && 1.21e-02 & 1.59 && 1.76e-01 & 0.80 && 1.46e-01 & 0.93 \\
&1.25e-01 && 4.53e-03 & 1.41 && 9.55e-02 & 0.88 && 7.21e-02 & 1.01 \\
&6.25e-02 && 1.83e-03 & 1.30 && 4.96e-02 & 0.94 && 3.55e-02 & 1.02 \\
&3.12e-02 && 7.79e-04 & 1.24 && 2.52e-02 & 0.97 && 1.76e-02 & 1.01 \\ \hline
\multirow{5}{*}{1} 
&5.00e-01 && 2.09e-02 & -    && 8.28e-02 & -    && 7.75e-02 & - \\
&2.50e-01 && 5.16e-03 & 2.02 && 2.15e-02 & 1.94 && 2.05e-02 & 1.92 \\
&1.25e-01 && 1.11e-03 & 2.21 && 5.65e-03 & 1.93 && 5.17e-03 & 1.99 \\
&6.25e-02 && 2.50e-04 & 2.15 && 1.45e-03 & 1.96 && 1.29e-03 & 2.01 \\
&3.12e-02 && 5.91e-05 & 2.08 && 3.67e-04 & 1.98 && 3.21e-04 & 2.00\\ \hline
\multirow{5}{*}{2} 
&5.00e-01 && 3.67e-03 & -    && 2.05e-02 & -    && 1.72e-02 & - \\
&2.50e-01 && 4.85e-04 & 2.92 && 2.44e-03 & 3.07 && 2.21e-03 & 2.96 \\
&1.25e-01 && 6.77e-05 & 2.84 && 2.79e-04 & 3.13 && 2.80e-04 & 2.98 \\
&6.25e-02 && 8.99e-06 & 2.91 && 3.31e-05 & 3.08 && 3.53e-05 & 2.99 \\
&3.12e-02 && 1.15e-06 & 2.96 && 4.02e-06 & 3.04 && 4.43e-06 & 2.99 \\ \hline
\multirow{5}{*}{3}  
&5.00e-01 && 8.27e-04 & -    && 3.27e-03 & -    && 3.01e-03 & - \\
&2.50e-01 && 5.06e-05 & 4.03 && 2.14e-04 & 3.93 && 2.01e-04 & 3.90\\
&1.25e-01 && 2.88e-06 & 4.13 && 1.39e-05 & 3.95 && 1.27e-05 & 3.98 \\
&6.25e-02 && 1.68e-07 & 4.10 && 8.83e-07 & 3.98 && 7.92e-07 & 4.00\\
&3.12e-02 && 1.01e-08 & 4.06 && 5.56e-08 & 3.99 && 4.94e-08 & 4.00\\
\bottomrule
\end{tabular}
\caption{Experiment 1: Convergence history of the numerical solution for the initial condition, computed using the HDG method \eqref{eq:HDG-hodge-laplacian}}
\label{tab:exp1:initialcondition}
\end{table}

The convergence results for the symplectic Hamiltonian HDG method of degree $k$, combined with symplectic sEPRK integrators of order $k+2$ for $k = 0, 1, 2, 3$, at the final time $T = 0.5$, are reported in Table~\ref{tab:label_tab}. The time step is chosen as $\Delta t = C h$, with $C = 0.1/(k+1)$. As observed in the initialization test, the method achieves optimal convergence of order $k+1$ for all variables: the geopotential height $\phi_h$, the horizontal velocity $\velocity_h$, and the auxiliary variable $\displacement_h$.  The Butcher tableau coefficients for the sEPRK integrators can be found in \cite{McLachlan1992, Hairer1993, Ruth1983}.
\begin{table}[h]
\centering
\scriptsize
\begin{tabular}{@{}lcl@{\hskip .2in}cc@{\hskip .1in}c@{\hskip .1in}cc@{\hskip .1in}c@{\hskip .1in}cc}
\toprule
& & &\multicolumn{2}{c}{$\phi_h$} && \multicolumn{2}{c}{$\velocity_h$} && \multicolumn{2}{c}{$\displacement_h$}  \\
\cmidrule(lr){4-5} \cmidrule(lr){7-8} \cmidrule(lr){10-11} 
$k$ 
& $h$     &&error     &e.o.c.&&  error   &e.o.c.&&  error   & e.o.c\\
\midrule
\multirow{5}{*}{0} 
&5.00e-01 && 4.27e-01 & - && 4.65e-01 & - && 2.69e-01 & - \\
&2.50e-01 && 3.81e-01 & 0.16 && 3.52e-01 & 0.4 && 1.44e-01 & 0.90\\
&1.25e-01 && 2.74e-01 & 0.48 && 2.63e-01 & 0.42 && 6.67e-02 & 1.11 \\
&6.25e-02 && 1.70e-01 & 0.69 && 1.90e-01 & 0.47 && 2.74e-02 & 1.28 \\
&3.12e-02 && 9.56e-02 & 0.83 && 1.10e-01 & 0.79 && 1.21e-02 & 1.18 \\ \hline
\multirow{5}{*}{1} 
&5.00e-01 && 1.25e-01 & - && 2.57e-01 & - && 6.27e-02 & - \\
&2.50e-01 && 3.33e-02 & 1.91 && 1.15e-01 & 1.16 && 1.25e-02 & 2.33 \\
&1.25e-01 && 6.77e-03 & 2.30&& 4.17e-02 & 1.46 && 2.98e-03 & 2.07 \\
&6.25e-02 && 1.10e-03 & 2.62 && 9.89e-03 & 2.08 && 1.16e-03 & 1.36 \\
&3.12e-02 && 2.16e-04 & 2.34 && 1.31e-03 & 2.91 && 2.79e-04 & 2.05 \\ \hline
\multirow{5}{*}{2} 
&5.00e-01 && 2.10e-02 & - && 8.07e-02 & - && 1.11e-02 & - \\
&2.50e-01 && 2.08e-03 & 3.34 && 1.63e-02 & 2.31 && 8.80e-04 & 3.66 \\
&1.25e-01 && 1.84e-04 & 3.50&& 3.30e-03 & 2.30&& 1.53e-04 & 2.52 \\
&6.25e-02 && 2.62e-05 & 2.81 && 3.53e-04 & 3.22 && 4.88e-05 & 1.65 \\
&3.12e-02 && 2.72e-06 & 3.26 && 2.28e-05 & 3.94 && 3.63e-06 & 3.74 \\  \hline
\multirow{5}{*}{3} 
&5.00e-01 && 3.55e-03 & - && 1.57e-02 & - && 1.67e-03 & - \\
&2.50e-01 && 1.75e-04 & 4.34 && 1.79e-03 & 3.13 && 5.70e-05 & 4.87 \\
&1.25e-01 && 7.55e-06 & 4.53 && 1.43e-04 & 3.65 && 1.13e-05 & 2.33 \\
&6.25e-02 && 5.73e-07 & 3.72 && 4.35e-06 & 5.04 && 9.94e-07 & 3.51 \\
&3.12e-02 && 2.97e-08 & 4.26 && 2.27e-07 & 4.25 && 3.02e-08 & 5.03 \\
\bottomrule
\end{tabular}
\caption{Experiment 1: Convergence history of the numerical approximations to the shallow water equations \eqref{eq:LSWEs2}, computed using the HDG scheme \eqref{eq:HDG-velocity-w} combined with  sEPRK of order $k+2$.}
    \label{tab:label_tab}
\end{table}
\subsection*{Experiment 2: Conservation properties}
In this experiment, we simulate the impact of a wavefront against a pier column and its subsequent interaction with the surrounding fluid. The computational domain is defined as $\Omega = \Omega_1 \setminus \Omega_2$, where $\Omega_1 = (-10, 10) \times (-10, 10)$ and $\Omega_2 = \mathcal{B}(c, r)$ is a hole of radius $r = 1$ centered at $c = (3, 0)$. We impose periodic boundary conditions on the outer boundary $\partial \Omega_1$, and wall-type boundary conditions on the internal boundary $\partial \Omega \setminus \partial \Omega_1$. The physical parameters for this experiment are $\Phi = 1$, $f = 0.5$, $\gamma = 0$, $\vect{s}=0$ and $g = 1$. The initial condition is chosen to model a right-moving wavefront centered at $x = -5$, and it is given by
\begin{alignat*}{4}
\phi_0(x,y) &= 1 + \exp\left(-(x+5)^2/2\right), \quad
\velocity_0(x, y) &= [\exp(-(x+5)^2/2),0]^\top.
\end{alignat*}
For the spatial discretization, we use HDG methods with polynomial spaces of degree $k = 2$ for all variables, with HDG stabilization $\tau=1$. Time integration is performed using the implicit midpoint method with $\Delta t=C h$ with $C=5 \times 10^{-2}.$ 

In Figure~\ref{fig:depth3D0}, we observe the wavefront propagation in the $x$-direction until it collides with the internal column, which is modeled by the wall-type boundary condition of the internal boundary.  The collision generates outward-propagating waves around the column and induces a rotational effect. After impact, there is an expanding wave pattern that results from interaction with the internal boundary.
\begin{figure}[htbp]
 \centering
 \subfigure{\includegraphics[width=54mm]{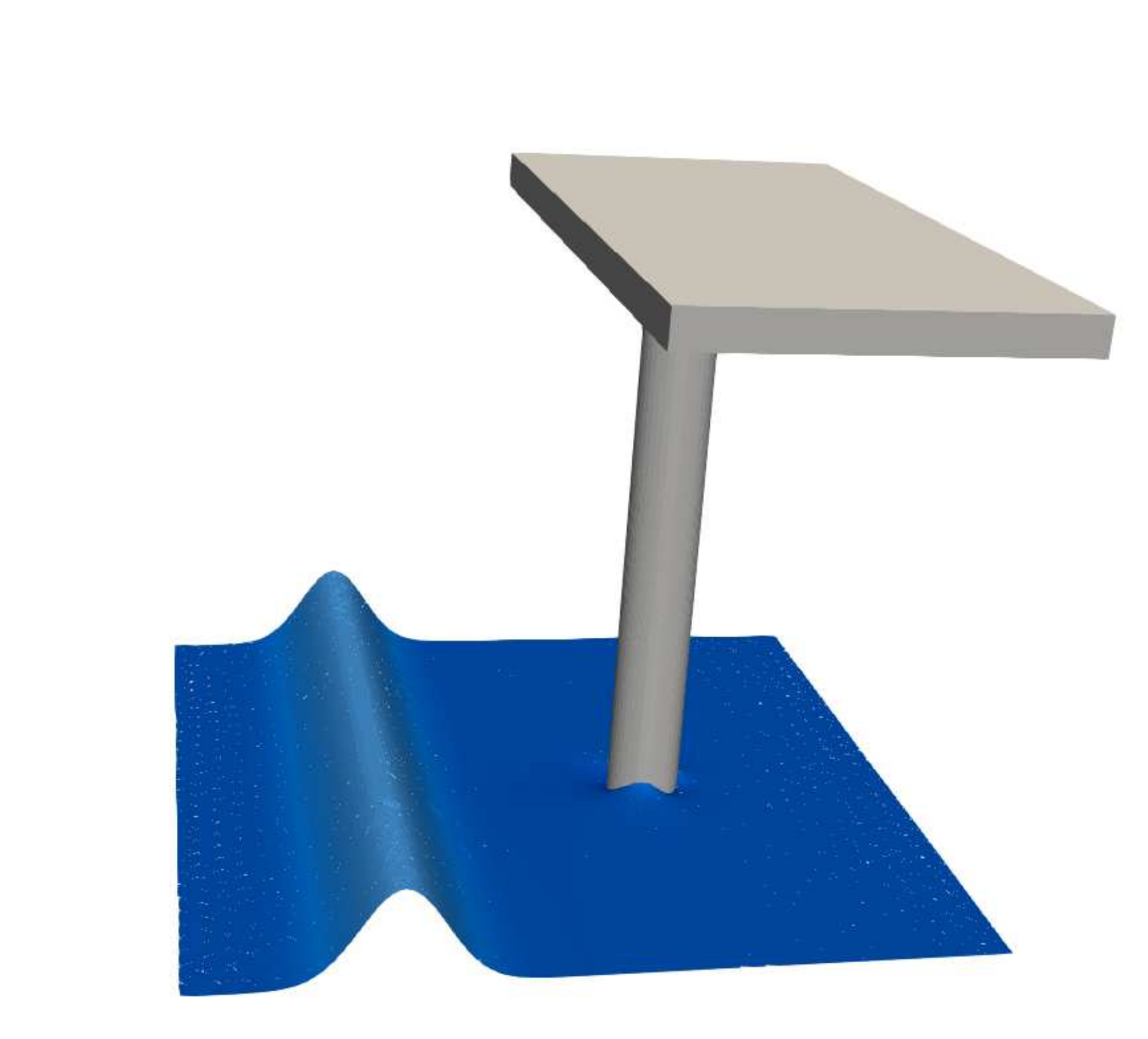}}
 \subfigure{\includegraphics[width=54mm]{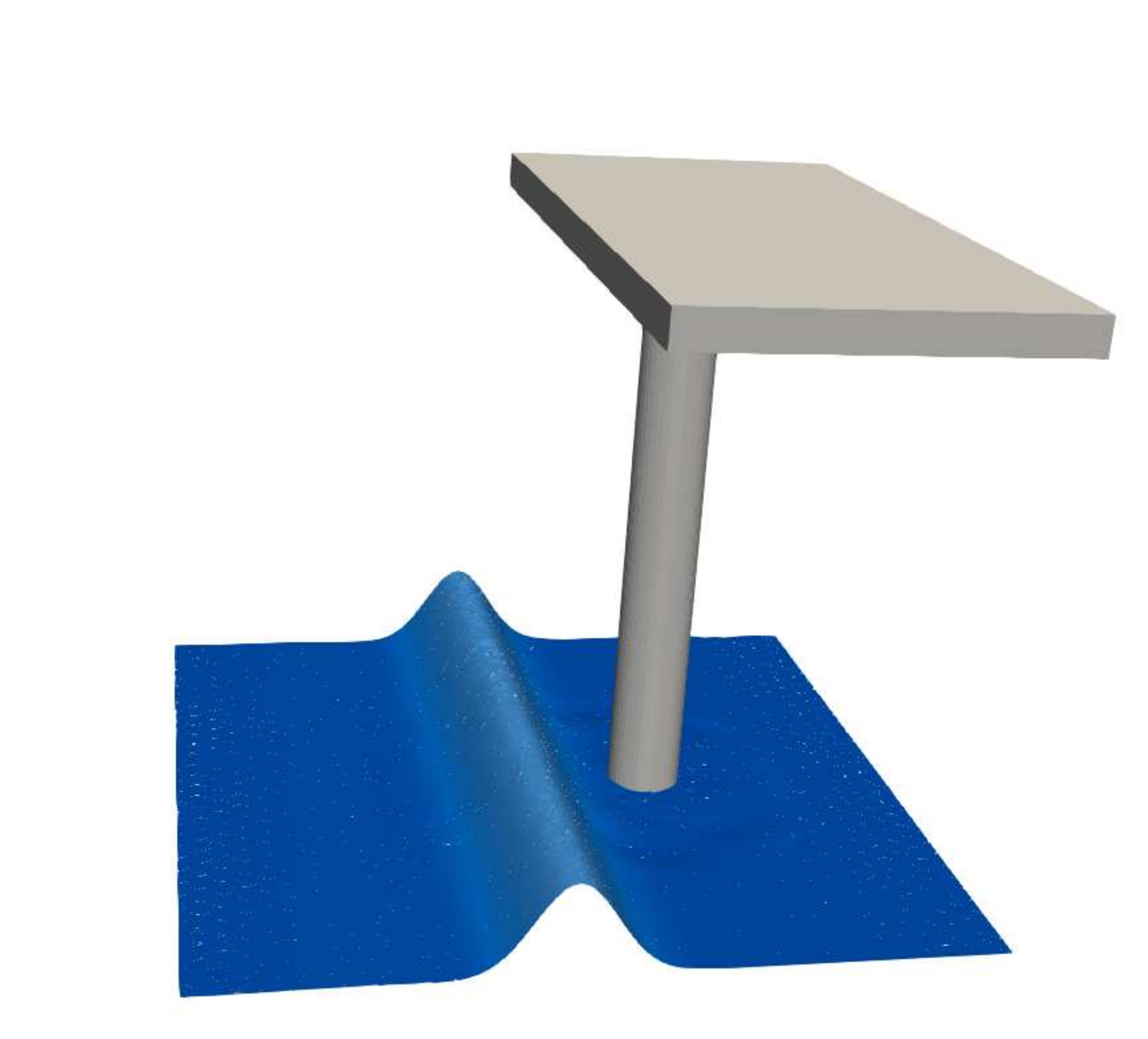}}
 \subfigure{\includegraphics[width=54mm]{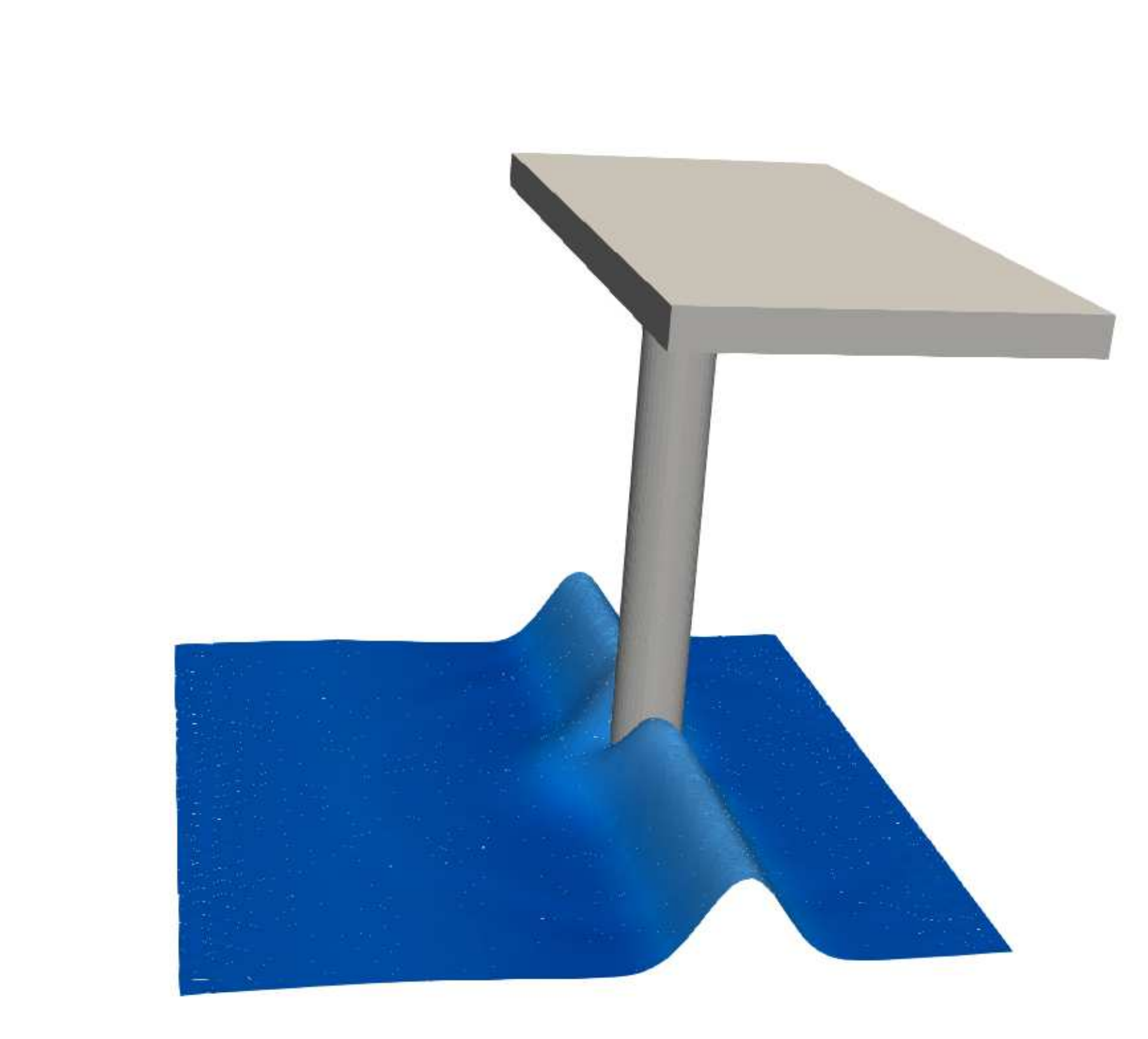}}\\
 \subfigure{\includegraphics[width=54mm]{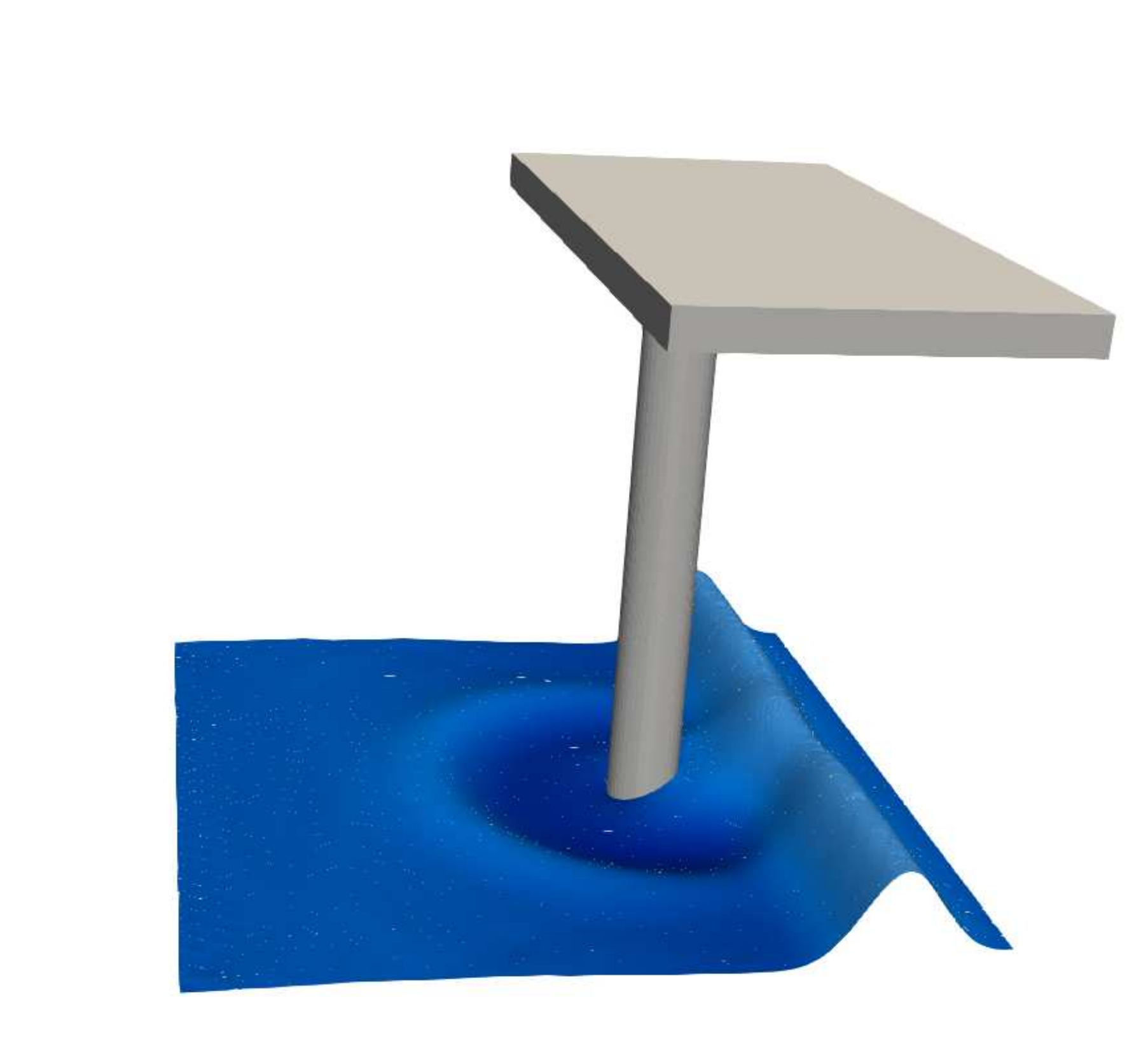}}
 \subfigure{\includegraphics[width=54mm]{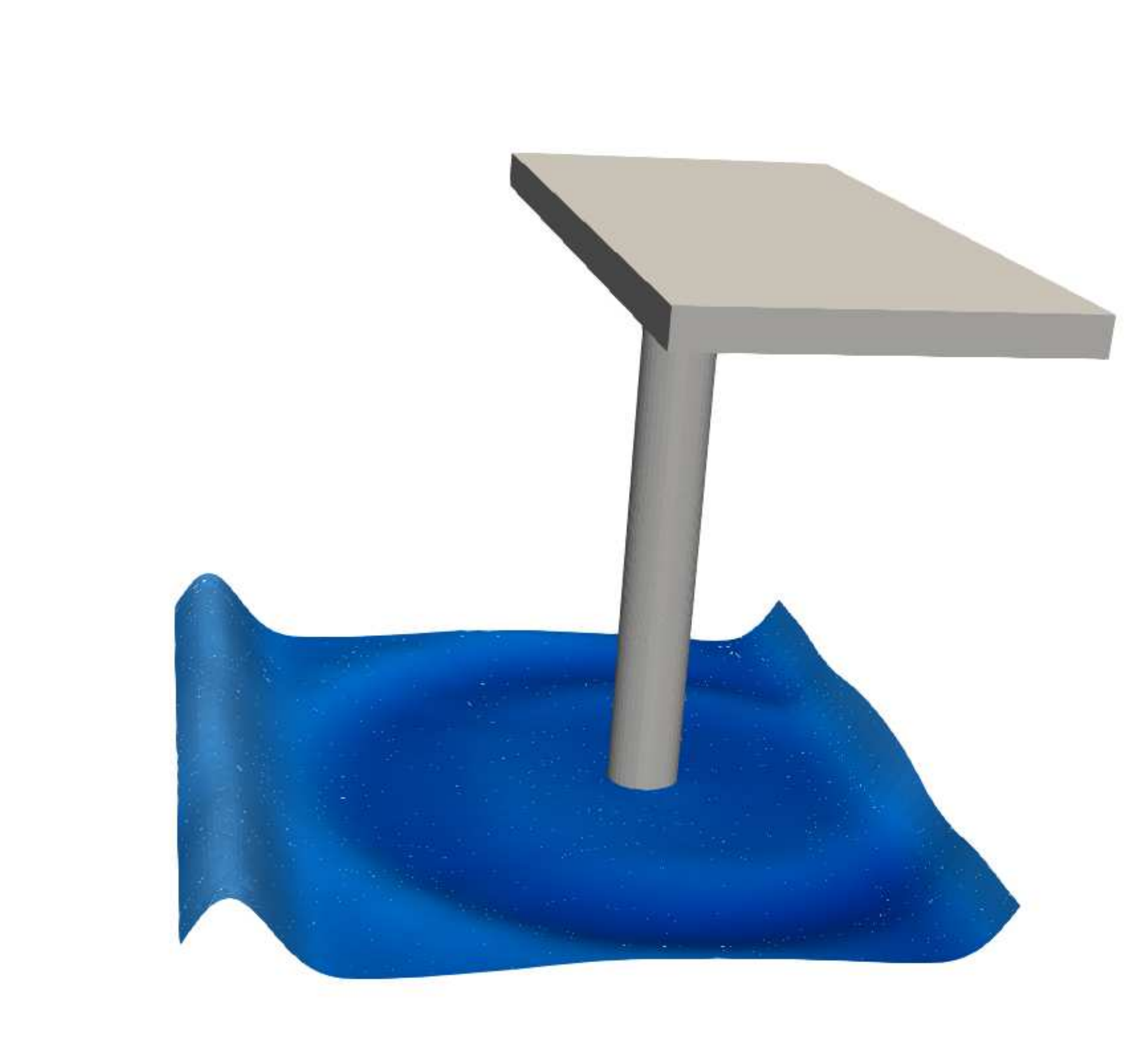}}
 \subfigure{\includegraphics[width=54mm]{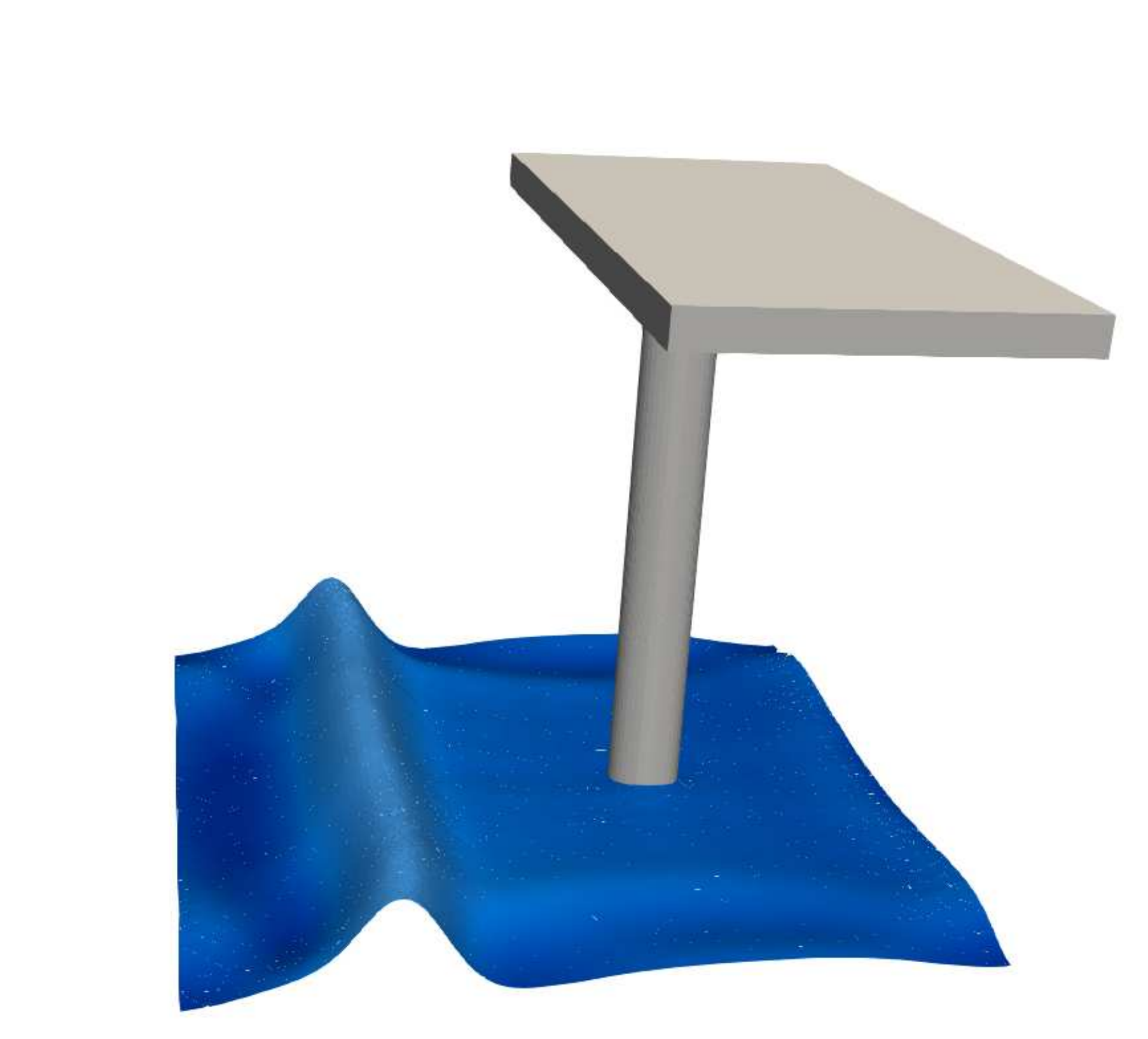}}
\caption{Experiment 2: Geopotential height $\phi_h$ at times $t = 0$, $4$, $8$, $12$, $16$, $20$ (3D view), computed using the HDG method \eqref{eq:HDG-velocity-w} of degree $k=2$,  mesh size $h = 1.25 \times 10^{-1}$, and the symplectic implicit midpoint scheme with time step $\Delta t = 6.25 \times 10^{-3}$.
} 
\label{fig:depth3D0}
\end{figure}
Furthermore, Figure~\ref{fig:conservation} displays the evolution of the mass, energy, linear and angular momentum, vorticity, and potential vorticity, for three discretization levels: $h$, $h/2$, and $h/4$, with mesh size $h = 0.5$.  The mass exhibits small oscillations around zero at machine precision. The energy is conserved exactly in the three discretizations. The first and second components of the linear momentum, as well as the angular momentum, exhibit a convergence behavior. Both vorticity and potential vorticity present small oscillations around zero with amplitudes below $10^{-2}$.

\begin{figure}[htbp]
 \centering
 \subfigure{\includegraphics[width=0.75\linewidth]{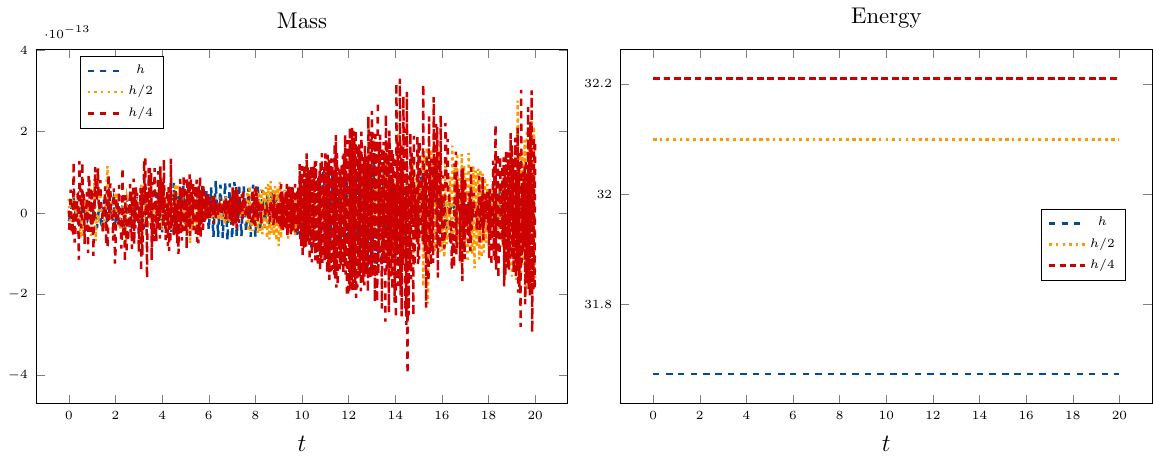}}\\ \vspace{-0.5cm}
 \subfigure{\includegraphics[width=0.75\linewidth]{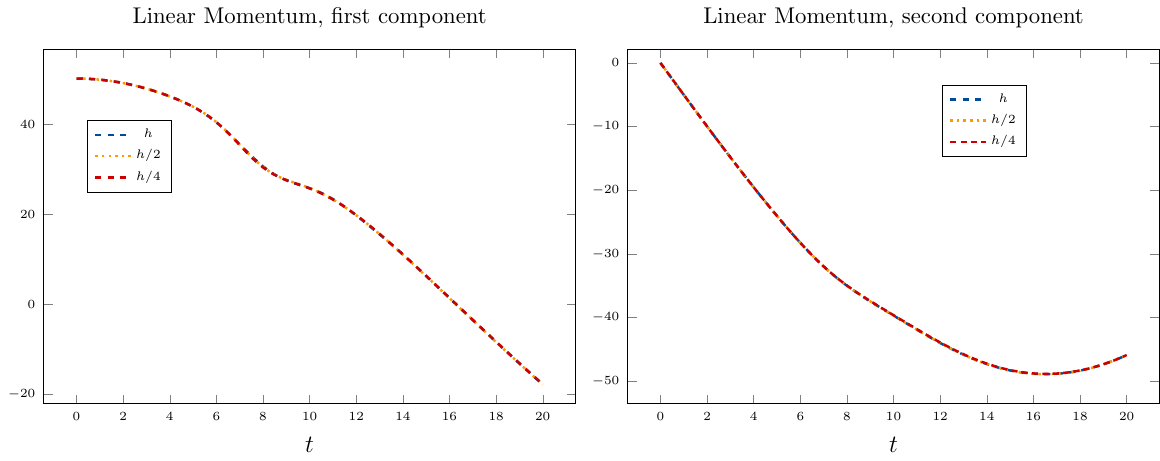}}\\ \vspace{-0.5cm}
 \subfigure{\includegraphics[width=0.75\linewidth]{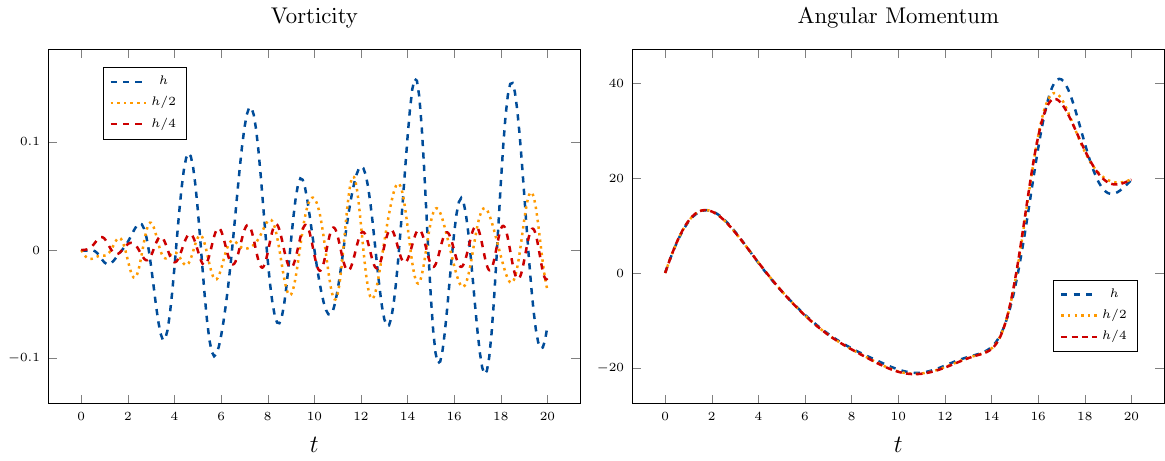}}\\ \vspace{-0.5cm}
 \subfigure{\includegraphics[width=65mm]{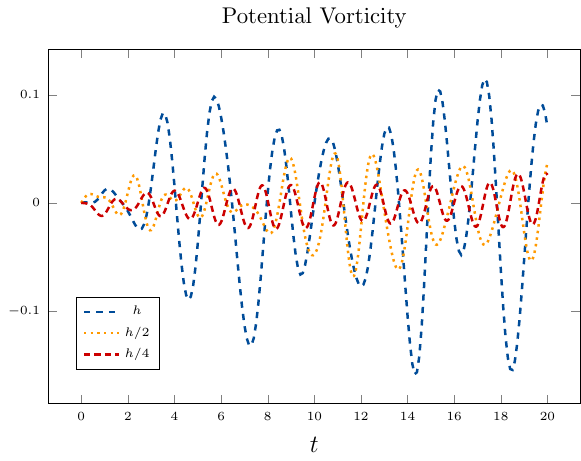}}
       \caption{Experiment 2: Mass (top, left), total energy (top, right), the first component of the  linear momentum (second row, left), the second component of the  linear momentum (second row, right), vorticity (third row, left), angular momentum (third row, right), and potential vorticity (bottom).} \label{fig:conservation}
\end{figure}

\subsection*{Experiment 3: Influence of bathymetry}
We now incorporate the effect of variable bathymetry, represented by $\phi_s = \phi_s(\xb)$. We consider the external force $\vect{s} = \Phi\grad \phi_s$ on the right-hand side of the momentum equation~\eqref{eq:LSWE2:b} in the $(\velocity -\displacement)$ formulation. It can be shown that the resulting system is a Hamiltonian system by modifying the Hamiltonian introduced in Section~\ref{sec:section2} (including the discrete Hamiltonian in Section \ref{sec:section4}). This is achieved by considering
\begin{alignat*}{4}
\Hamiltonian_2[\velocity,\displacement] &= \frac{1}{2} \int_\Omega \left(  |\divergence \displacement|^2 + \Phi\, |\velocity|^2 - 2\, \phi_s\, \divergence \displacement \right).
\end{alignat*}
Including variable bathymetry can capture terrain effects on fluid motion (see, \textit{e.g.}, \cite{topography1,GarciaNavarro2019}). To illustrate this effect, we perform a numerical experiment in the domain $\Omega = (-20, 10) \times (-5, 5)$, with wall-type boundary conditions. The bathymetry is modeled as an inclined plane beginning at $x = 0$, with three localized mounds placed along the slope to analyze their localized influence on the geopotential height and horizontal velocity. The bathymetry for this experiment is 
\[
\phi_s(x,y) = 
\begin{cases}
-1.1 + \dfrac{3}{5} \left[ \beta(x-5,y) + \beta(x-5,y-3) + \beta(x-5,y+3) \right], & \text{if } x \geq 0, \\[0.5em]
0, & \text{otherwise},
\end{cases}
\]
with $\beta(x,y) = \exp(-2x^2)\exp(-2y^2).$  We set the physical parameters as $\Phi = 1$, $f = 0.1$, $\gamma = 0$, and $g = 1$.  The initial states consist of a Gaussian pulse and a zero vector. 
\[
\phi_0(x,y) = 10\,\exp(-2y^2)\exp(-2(x+5)^2), \qquad
\velocity_0(x,y) = [ 0, 0]^\top.
\]

\begin{figure}[htbp]
\centering
\subfigure{\includegraphics[width=64mm]{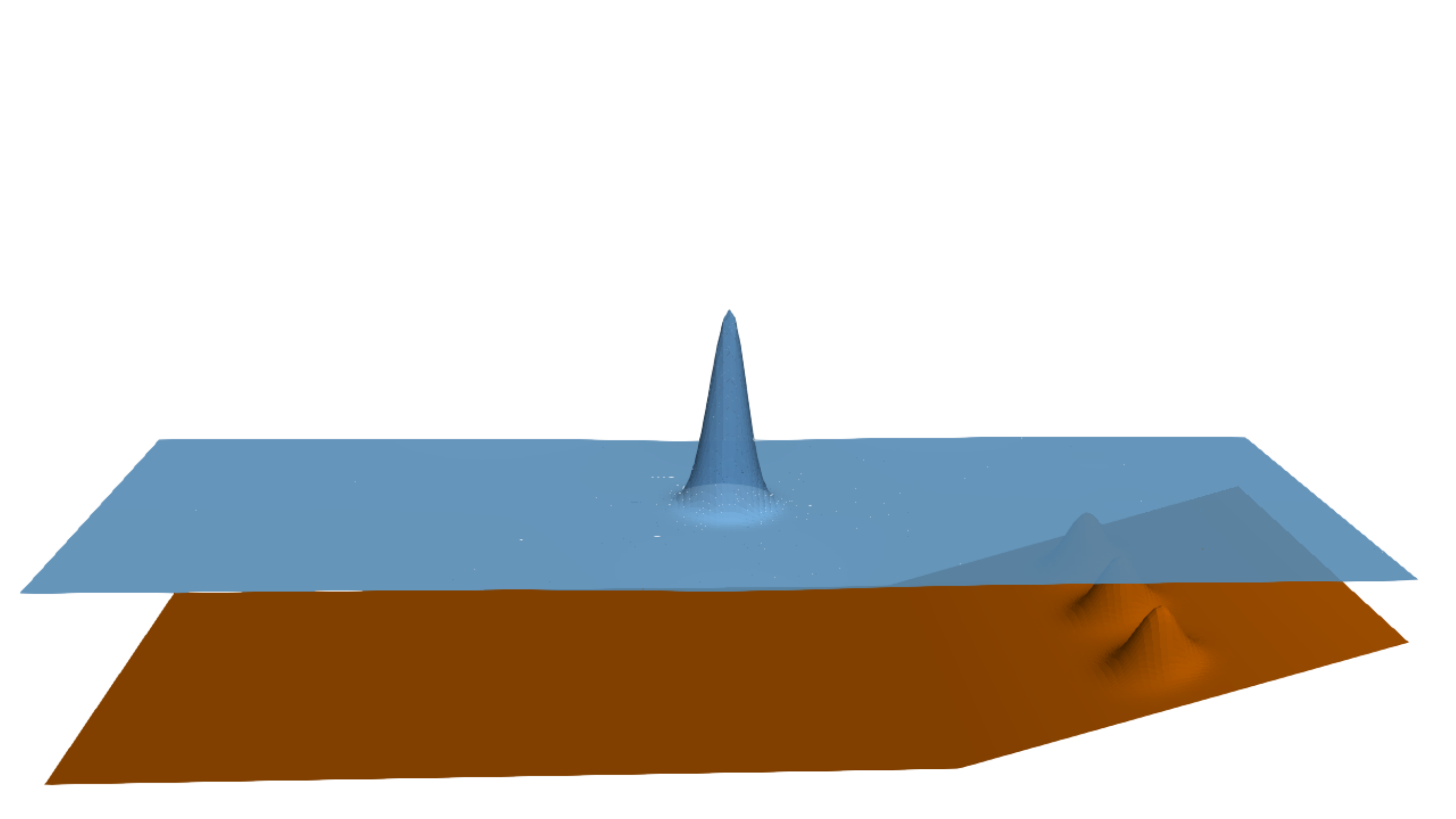}}
\subfigure{\includegraphics[width=64mm]{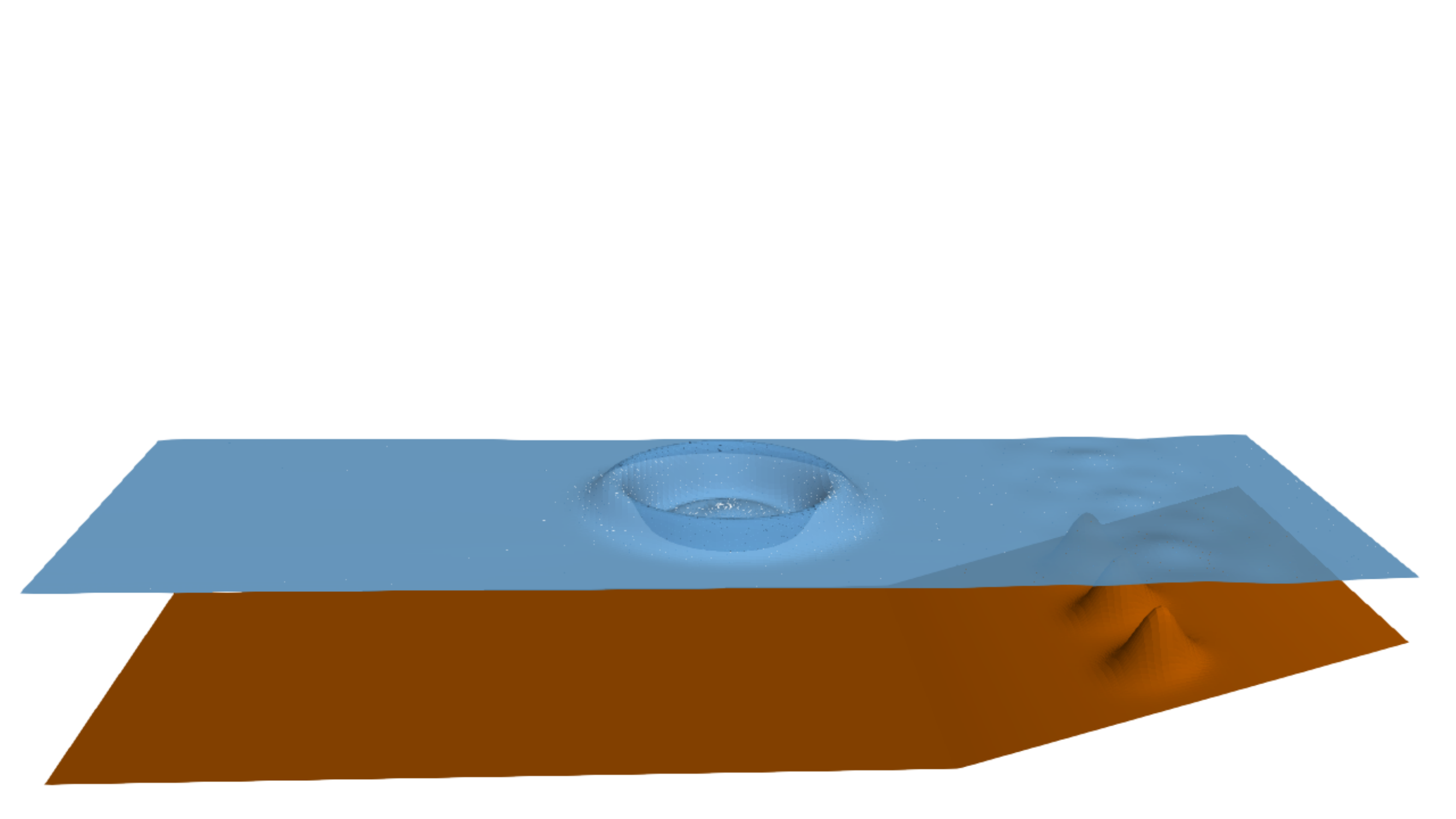}}\\ \vspace{-0cm}
\subfigure{\includegraphics[width=64mm]{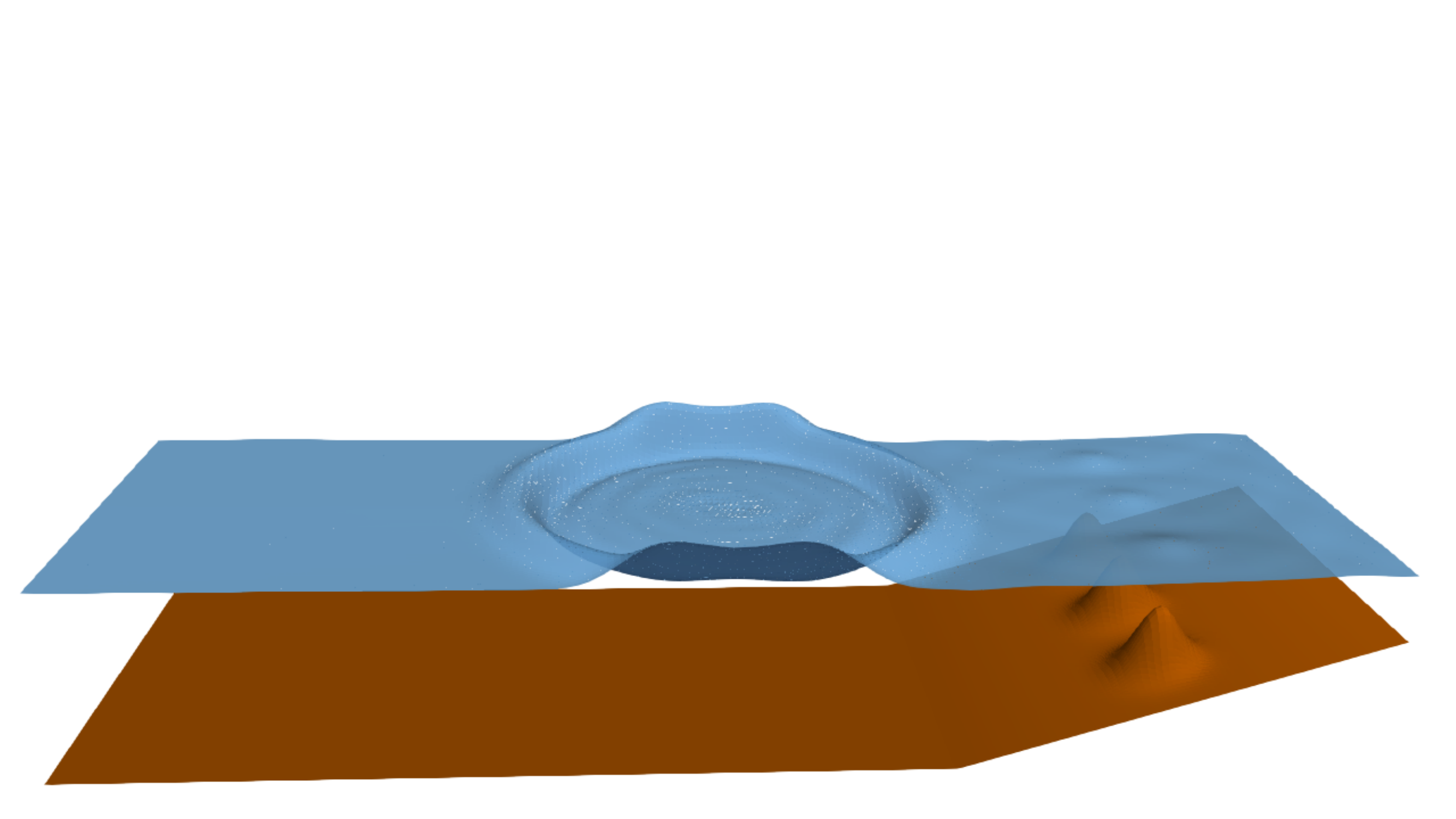}}
\subfigure{\includegraphics[width=64mm]{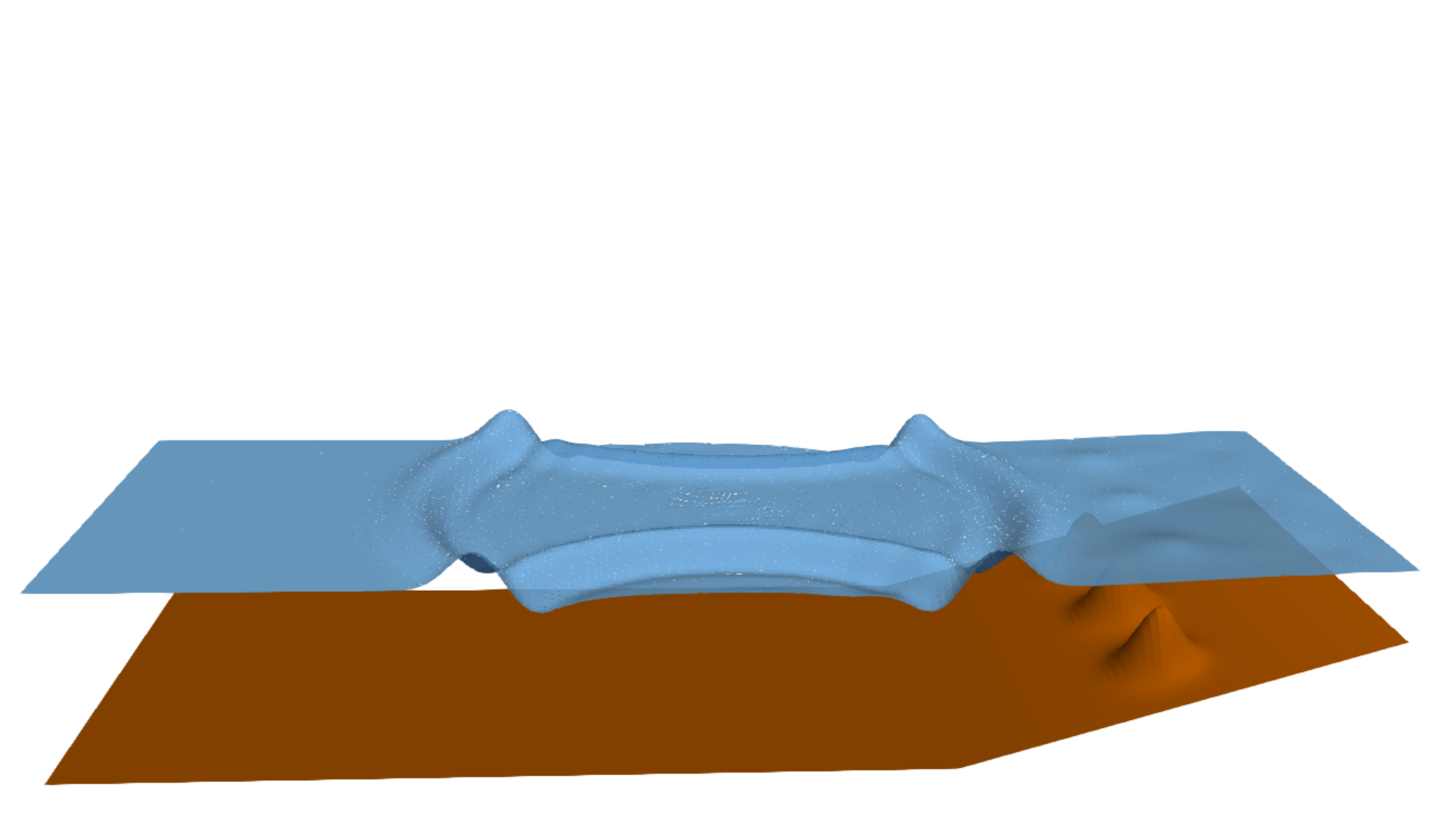}}\\ \vspace{-0cm}
\subfigure{\includegraphics[width=64mm]{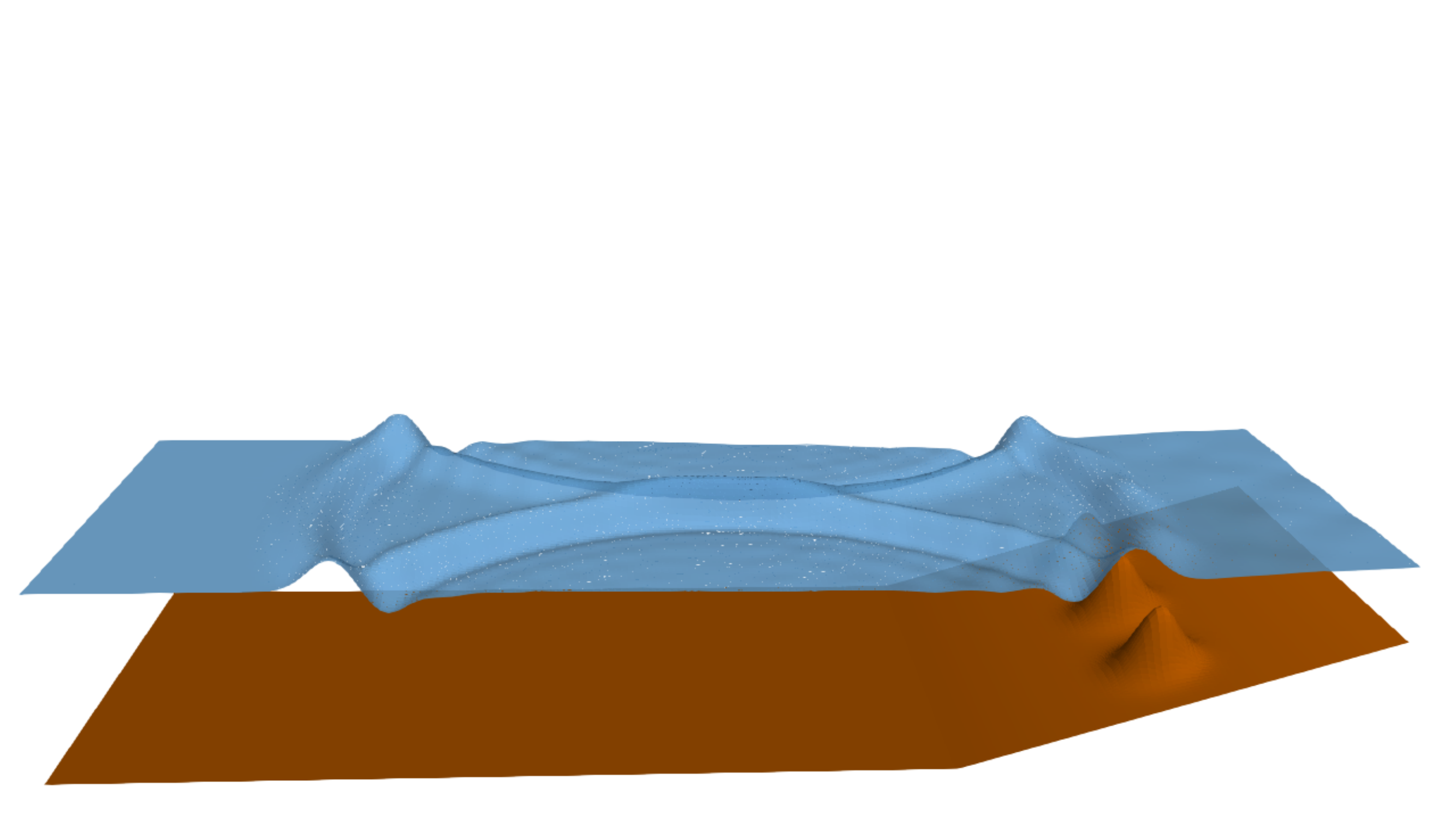}}
\subfigure{\includegraphics[width=64mm]{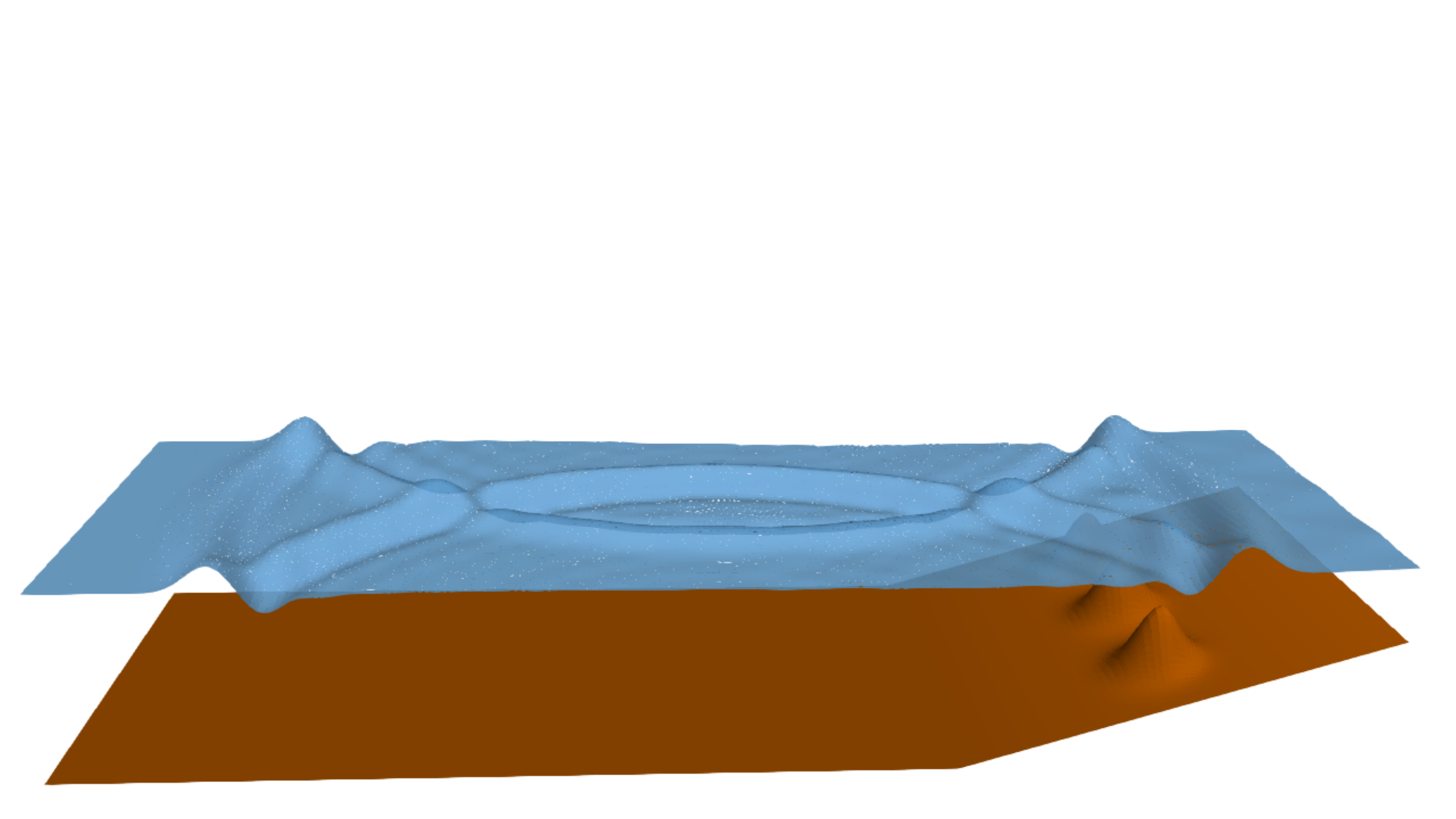}}
\caption{Experiment 3: Geopotential height $\phi_h$ at times $t = 0$, $2$, $4$, $6$, $8$, $10$,  (blue color) and bathymetry (color brown). The simulation is performed using the HDG method \eqref{eq:HDG-velocity-w} with polynomial degree $k=1$, mesh size $h = 10 \times 10^{-3}$, and the symplectic implicit midpoint scheme with time step $\Delta t = 5 \times 10^{-4}.$} \label{fig:topodepth3D}
\end{figure}
\begin{figure}[htbp]
 \centering
 \subfigure{\includegraphics[width=0.75\linewidth]{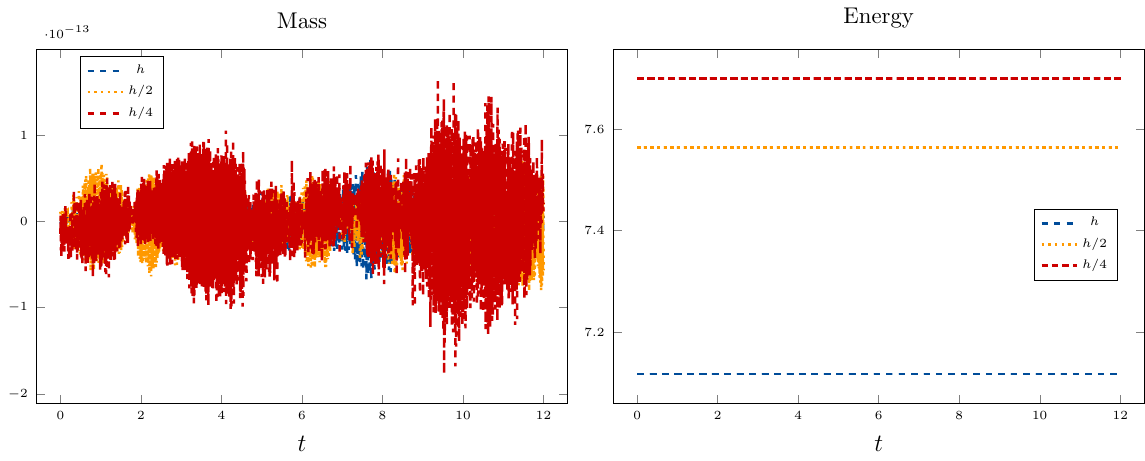}}\\ \vspace{-0.5cm}
 \subfigure{\includegraphics[width=0.75\linewidth]{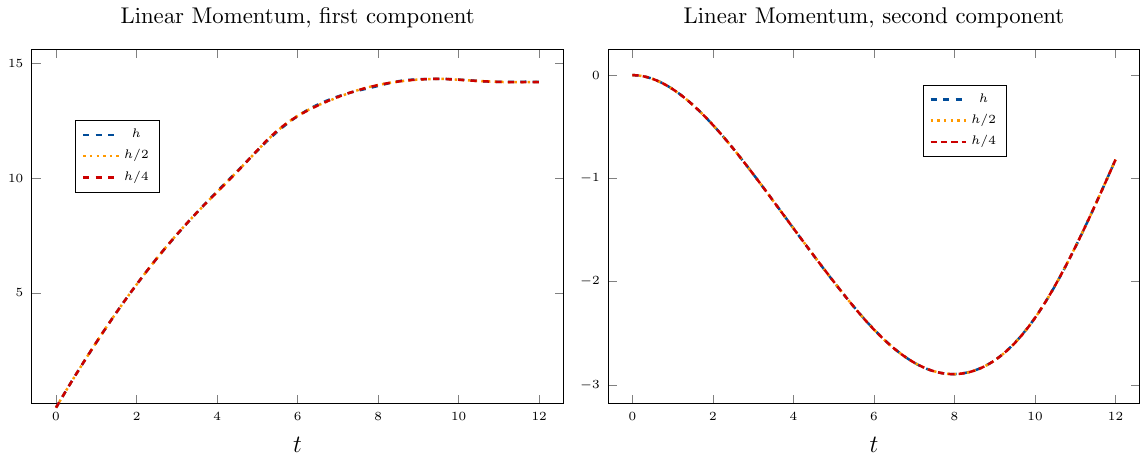}}\\ \vspace{-0.5cm}
 \subfigure{\includegraphics[width=0.75\linewidth]{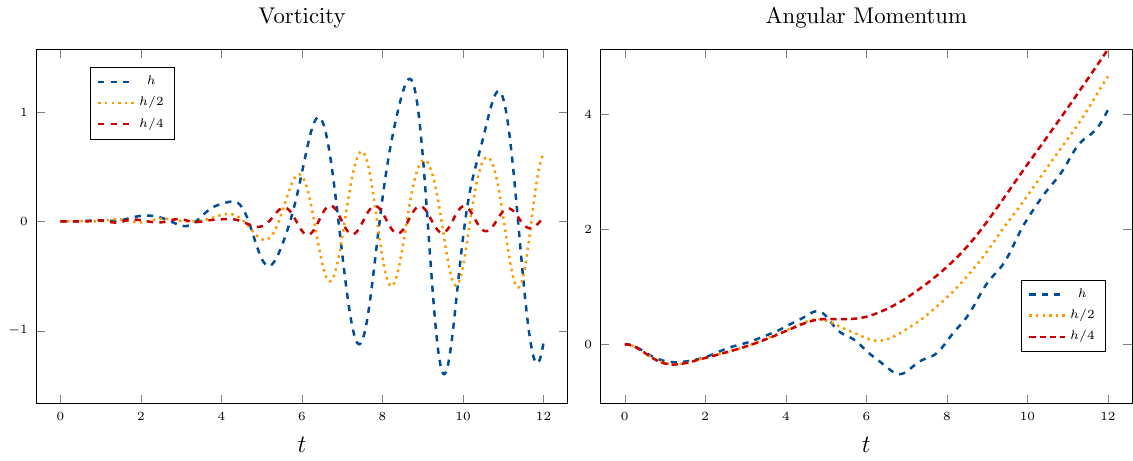}}\\ \vspace{-0.5cm}
 \subfigure{\includegraphics[width=65mm]{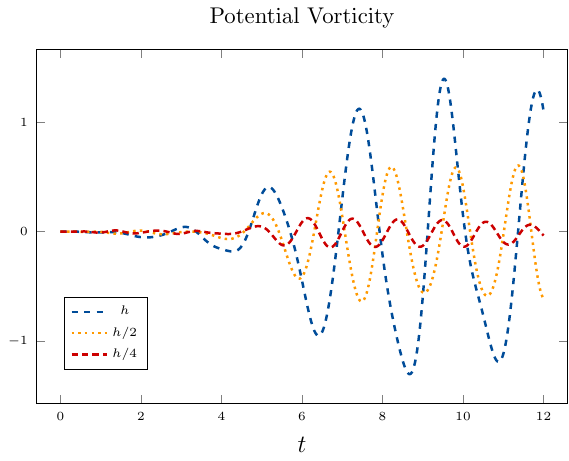}}
       \caption{Experiment 3: Mass (top, left), total energy (top, right), the first component of the linear momentum (second row, left), the second component of the linear momentum (second row, right), vorticity (third row, left), angular momentum (third row, right), and potential vorticity (bottom).}  \label{fig:Conservedquantities:Topography}
\end{figure}

For the HDG method, we use polynomials of degree $k = 1$ and a stabilization parameter of $\tau = 1$. Time integration is performed using the symplectic midpoint scheme. Figure~\ref{fig:topodepth3D} shows 3D snapshots of the geopotential height $\phi_h$ (blue color) over the bathymetry $\phi_s$ (brown color). The initial pulse propagates symmetrically but slows down as it climbs the slope. The three mounds induce localized vortices that generate a rotating motion of the flow. Additionally, Figure~\ref{fig:Conservedquantities:Topography} shows the evolution of various physical quantities for three discretizations: $h$, $h/2$, and $h/4$, with mesh size $h = 4\times 10^{-2}$ and $\Delta t= Ch$ with $C=5\times 10^{-2}$. Mass oscillates around zero with machine precision. Energy is conserved exactly across all discretizations. Both linear and angular momentum components show convergence behavior. Vorticity and potential vorticity exhibit small oscillations around zero.
\section{Conclusions}\label{sec:section7}
We developed Symplectic Hamiltonian HDG methods for the linearized shallow water equations. By introducing an auxiliary variable $\displacement$ satisfying $\phi = -\divergence \displacement$, we reformulated the system as a  Hamiltonian PDE for the $(\velocity-\displacement)$ formulation. Based on this formulation, we proposed HDG methods that preserve a discrete Hamiltonian structure, in which the numerical energy includes an additional quadratic term involving the numerical trace (see Theorem~\ref{teo:HamiltoniansystemHDG}). We also introduced an initialization procedure based on solving a vector Laplacian equation. We rewrote the HDG formulation in matrix form and provided an alternative proof of its Hamiltonian structure in the context of finite-dimensional dynamical systems. Additionally, for time discretization, we introduced the symplectic sDIRK and sEPRK methods for the conservation of discrete energy. Numerical experiments indicated optimal convergence rates and provided evidence of long-time conservation of mass, vorticity, and total energy.

In comparison to earlier energy-preserving schemes, such as those based on finite differences \cite{Arakawa1981, Arakawa1997}, including Poisson or Nambu bracket discretizations \cite{Salmon, Salmon2, Salmon2009}, our method offers potential advantages in terms of higher-order accuracy and mesh flexibility. Compared to discontinuous Hamiltonian finite element methods based on Poisson brackets \cite{XuvanderVegtBokhove08,12_Cotter,11_Cotter_MR3874575,MR4038160}, as well as entropy-stable schemes \cite{1_Kieran,3_Kieran,5_GASSNER2016291,6_2_WINTERMEYER2018447}, our approach achieves a significant reduction in globally coupled unknowns through hybridization, particularly in the context of implicit discretizations. Moreover, unlike the IMEX HDG-DG and HDG methods presented in \cite{10_Dawson,11_ThanBui,Bui-Thanh2016}, which do not explicitly adopt a Hamiltonian framework for energy conservation, our formulation ensures the preservation of the Hamiltonian structure at the spatial discretization level.

Future work will focus on extending the proposed methods to the nonlinear case. This extension presents new challenges, particularly in designing a reformulation that preserves the Hamiltonian structure while properly handling the nonlinear numerical fluxes arising in the spatial discretization. Furthermore, future research will focus on the study of a priori error estimates for the numerical methods proposed in this paper, similar to that presented in \cite{MR4609879} for Maxwell’s equations.
\bibliography{bibliography}
\bibliographystyle{plain}

\end{document}